\documentclass[10pt]{amsart}
\usepackage{amsthm,amsmath,amssymb}

\global\overfullrule10pt

\theoremstyle{plain}
\newtheorem{theorem}{Theorem}
\newtheorem{corollary}[theorem]{Corollary}

\newtheorem{proposition}[theorem]{Proposition}
\newtheorem*{theorem*}{Theorem}
\newtheorem*{corollary*}{Corollary}
\newtheorem*{proposition*}{Proposition}
\theoremstyle{definition}

\theoremstyle{remark}
\newtheorem{remark}[theorem]{Remark}

\newcommand\CC{{\mathbf C}}

\newcommand\DD{{\mathbf D}}
\newcommand\BB{{\mathbf B}}
\newcommand\spr[1]{\langle#1\rangle}
\newcommand\hol{^{\text{hol}}}
\newcommand\harm{^{\text{harm}}}
\newcommand\mh{_{\text{Mh}}}
\renewcommand\Re{\operatorname{Re}}

\newcommand\oz{\overline z}
\newcommand\ozeta{\overline\zeta}
\newcommand\oa{\overline a}
\newcommand\ob{\overline b}
\newcommand\oc{\overline c}
\newcommand\Bn{{\BB^n}}
\newcommand\pBn{{\partial\Bn}}
\newcommand\bbd{{\overline{\Bn\times\Bn}\setminus\operatorname{diag}\pBn}}

\newcommand\wt{\widetilde}
\newcommand\td{\wt\Delta}
\newcommand\Sz{_{\fam0Sz}}

\newcommand\FF[5]{{}_#1\!F_#2\Big(\begin{matrix}#3\\#4\end{matrix}\Big|#5\Big)}
\newcommand\FE[4]{F_#1\Big(\begin{matrix}#2\\#3\end{matrix}\Big|#4\Big)}
\newcommand\FD[3]{FD_1\Big(\begin{matrix}#1\\#2\end{matrix}\Big|#3\Big)}
\newcommand\Mh{$M$-harmonic }
\newcommand\LL{\mathcal L}
\newcommand\GG[2]{\Gamma\Big(\begin{matrix}#1\\#2\end{matrix}\Big)}
\newcommand\chpq{\mathcal H^{pq}}
\newcommand\bhpq{\mathbf H^{pq}}
\newcommand\Un{{U(n)}}
\newcommand\cdd[2]{#1\cdot\overline#2}
\newcommand\HH{{\mathcal H}}
\newcommand\dsph{\Delta_{\text{sph}}}
\newcommand\cR{\mathcal R}
\newcommand\cD{\mathcal D}
\newcommand\cS{\mathcal S}
\newcommand\oZ{{\overline Z}}

\makeatletter
\newcommand{\vast}{\bBigg@{4}}
\newcommand{\Vast}{\bBigg@{5}}
\makeatother

\renewcommand\[{\begin{equation}}
\renewcommand\]{\end{equation}}
\allowdisplaybreaks[4]

\begin{document}

\title[$M$-harmonic kernels]{$M$-harmonic reproducing kernels on the ball}
\author[M.~Engli\v s]{Miroslav Engli\v s}
\address{Mathematics Institute, Silesian University in Opava,
 Na~Rybn\'\i\v cku~1, 74601~Opava, Czech Republic {\rm and }
 Mathematics Institute, \v Zitn\' a 25, 11567~Prague~1,
 Czech Republic}
\email{englis{@}math.cas.cz}
\author{El-Hassan Youssfi}
\address{Aix-Marseille Universit\'e, I2M UMR CNRS~7373,
 39 Rue F-Juliot-Curie, 13453~Marseille Cedex 13, France}
\email{el-hassan.youssfi{@}univ-amu.fr}
\thanks{Research supported by GA\v CR grant no.~21-27941S
 and RVO funding for I\v CO~67985840.}
% 32-XX	Several complex variables and analytic spaces
%  32Axx Holomorphic functions of several complex variables
%   32A35 $H^p$-spaces, Nevanlinna spaces of functions in several complex variables
%   32A36 Bergman spaces of functions in several complex variables
% 31-XX Potential theory
%  31Cxx Generalizations of potential theory
%   31C05 Harmonic, subharmonic, superharmonic functions on other spaces
% 35-XX Partial differential equations
%  35Cxx Representations of solutions to partial differential equations
%   35C10 Series solutions to PDEs
% 33-XX Special functions
%  33Cxx Hypergeometric functions
%   33C55 Spherical harmonics
%   33C65 Appell, Horn and Lauricella functions
%   33C70 Other hypergeometric functions and integrals in several variables
\begin{abstract} Using the machinery of unitary spherical harmonics due to Koornwinder,
Folland and other authors, we~obtain expansions for the Szeg\"o and the weighted Bergman
kernels of $M$-harmonic functions, i.e.~functions annihilated by the invariant Laplacian
on the unit ball of the complex $n$-space. This yields, among others, an explicit formula
for the \Mh Szeg\"o kernel in terms of multivariable as well as single-variable
hypergeometric functions, and also shows that most likely there is no explicit (``closed'')
formula for the corresponding weighted Bergman kernels.
\end{abstract}

\maketitle

\section{Introduction}\label{sec1}
Recall that a function on the unit ball $\Bn$ of $\CC^n$, $n\ge1$, is~called \emph{Moebius-harmonic}
(or~\emph{invariantly harmonic}), or \emph{\Mh }for short, if~it is annihilated by the invariant Laplacian
\[ \td = 4(1-|z|^2) \sum_{j,k=1}^n (\delta_{jk}-z_j\oz_k)\frac{\partial^2}{\partial z_j\partial\oz_k}.
 \label{tTA} \]
It~is well known (see e.g.~Rudin~\cite{Ru}, Stoll~\cite{Stoll}, or Chapter~6 in Krantz~\cite{KraPDECA}) that
$\td$ commutes with biholomorphic self-maps (Moebius maps) of the ball:
$$ \td(f\circ\phi) = (\td f)\circ\phi, \qquad \forall f\in C^2(\Bn), \phi\in\operatorname{Aut}(\Bn); $$
and, accordingly, that \Mh functions possess the \emph{invariant mean-value property}: namely,
if $\td f=0$ and $z\in\Bn$, then $f(z)$ equals the mean value, with respect to the $\operatorname{Aut}(\Bn)$-invariant
measure $d\tau(z)=(1-|z|^2)^{-n-1}\,dz$, over any Moebius ball in $\Bn$ centered at~$z$ (and similarly
for spheres in the place of balls). It~follows by a standard argument that the point evaluations
$f\mapsto f(z)$ at any $z\in\Bn$ are continuous linear functionals on the subspace
$$ L^2\mh(\Bn) := \{f\in L^2(\Bn): f \text{ is $M$-harmonic}\}  $$
of all \Mh functions in $L^2(\Bn)$ (the \emph{\Mh Bergman space}), and therefore there exists a reproducing
kernel for $L^2\mh(\Bn)$ (the~\emph{\Mh Bergman kernel}), namely a function $K(x,y)$ on $\Bn\times\Bn$,
\Mh in both variables and such that
$$ f(x) = \int_\Bn f(y) K(x,y) \, dy \qquad \forall x\in\Bn, \forall f\in L^2\mh(\Bn).  $$
More generally, for any $s>-1$, one can consider the \emph{weighted \Mh Bergman space}
\[ L^2\mh(\Bn,(1-|z|^2)^s\,dz)  \label{tTB}  \]
(where $dz$ stands, throughout, for the Lebesgue measure of the appropriate dimension) and its
reproducing kernel $K_s(x,y)$ on $\Bn\times\Bn$ (the \emph{weighted \Mh Bergman kernel}).

For the analogous weighted Bergman spaces of {\sl holomorphic}, rather than $M$-harmonic, functions on~$\Bn$,
the reproducing kernels have been known explicitly for a long time: one~has
\[ K_s\hol(x,y) = \frac{\Gamma(n+s+1)}{\Gamma(s+1)\pi^n} (1-\spr{x,y})^{-n-s-1} . \label{tTC}  \]
Similarly, there are formulas expressing the {\sl harmonic} weighted Bergman
kernels $K_s\harm(x,y)$, $s>-1$, on~$\Bn$ in terms of Appell's hypergeometric function~$F_1$
of two variables \cite[Section~3]{Blaschke}:
\[ \begin{aligned}
K_s\harm(x,y) &= \frac{\Gamma(n+s+1)}{\Gamma(s+1)\pi^n} \FE1{n+s+1;n-1,n-1}{n-1}{z,\oz}, \\
&\hskip6em z=\spr{x,y}+i\sqrt{|x|^2|y|^2-|\spr{x,y}|^2} . \end{aligned}
 \label{tTD}   \]
For $n=1$, the unit ball $\Bn$ becomes just the unit disc $\DD:=\{z\in\CC:|z|<1\}$, and \eqref{tTA} reduces
to the multiple $(1-|z|^2)^2\Delta$ of the ordinary Laplacian~$\Delta$; thus \Mh and harmonic functions
coincide for $n=1$. Also, any harmonic function on $\DD$ can be written as $f+\overline g$ with
$f,g$ holomorphic and $g(0)=0$. Since holomorphic and conjugate-holomorphic functions are orthogonal
in any $L^2(\DD,(1-|z|^2)^s\,dz)$ except for the constants, it~follows that for $n=1$
$$ K_s = K_s\harm = 2\Re K_s\hol - \frac{\Gamma(n+s+1)}{\Gamma(s+1)\pi^n} .  $$
However, no~explicit formula seems to be available for the \Mh Bergman kernels on $\Bn$ for $n>1$.

If~we multiply the measures in \eqref{tTB} by the factor $2(s+1)$ and let $s\searrow-1$, it~is easily shown
that $2(s+1)(1-|z|^2)^s\,dz$ converges weakly to~$d\sigma$, the (unnormalized) surface measure on the
topological boundary $\pBn$ of~$\Bn$. As~a~limit of the weighted Bergman spaces \eqref{tTB} we thus obtain
the \emph{\Mh Hardy space} $H^2\mh(\Bn)$ of \Mh functions on~$\Bn$, whose reproducing kernel $K\Sz(x,y)$
--- the \emph{\Mh Szeg\"o kernel} --- is~the function on $\Bn\times\Bn$, \Mh in both variables, which satisfies
\[ f(x) = \int_\pBn f(\zeta) K\Sz(x,\zeta) \,d\sigma(\zeta), \qquad \forall f\in H^2\mh(\pBn),
 \forall x\in\Bn, \label{tTE}  \]
where, abusing the notation slightly, we~denote by the same letter $f$ also the radial boundary values
of $f$ on~$\pBn$, and similarly for $K\Sz(x,\cdot)$. In~the holomorphic case, one~again has the
explicit formula
\[ K\Sz\hol(x,y) = \frac{\Gamma(n)}{2\pi^n} (1-\spr{x,y})^{-n}  \label{tTL}  \]
for the ordinary Szeg\"o kernel of~$\Bn$, and similarly from \eqref{tTD} we get the \emph{harmonic Szeg\"o kernel}
\[ K\Sz\harm(x,y) = \frac{\Gamma(n)}{2\pi^n} \frac{1-|x|^2|y|^2}{(1-2\Re\spr{x,y}+|x|^2|y|^2)^n} \label{tTF} \]
for the harmonic case.

The harmonic and \Mh Szeg\"o kernels are intimately connected with the associated \emph{Poisson kernels}.
Namely, recall that the ordinary Poisson kernel
$$ P\harm(x,\zeta) = \frac{\Gamma(n)}{2\pi^n} \frac{1-|x|^2}{|x-\zeta|^{2n}}  $$
on $\Bn$ reproduces the values of a harmonic function in
the interior of $\Bn$ from its boundary values:
$$ f(x) = \int_\pBn f(\zeta) P\harm(x,\zeta) \,d\sigma(\zeta), \qquad\forall x\in\Bn,   $$
for any function $f$ harmonic on $\Bn$ and, say, continuous on the closure~$\overline{\Bn}$.
Comparing this with the (harmonic version~of) \eqref{tTE}, we~thus see that $P\harm(x,\cdot)$
are just the boundary values of $K\Sz\harm(x,\cdot)$; that~is, $K\Sz\harm(x,\cdot)$ is just
the harmonic extension of $P\harm(x,\cdot)$ from $\pBn$ into~$\Bn$,~or
$$ K\harm\Sz(x,y) = \int_\pBn P\harm(x,\zeta) P\harm(y,\zeta) \,d\sigma(\zeta).  $$
In~other words, $K\Sz\harm$ is just $P\harm$ extended from $\Bn\times\pBn$ to a function on
$\Bn\times\Bn$ harmonic in both variables.

Exactly the same argument shows that also the \Mh Poisson kernel
(called \emph{Poisson-Szeg\"o kernel} in~\cite{KraPDECA})
$$ P(x,\zeta) = \frac{\Gamma(n)}{2\pi^n} \frac{(1-|x|^2)^n}{|1-\spr{x,\zeta}|^{2n}}  $$
(cf.~\cite[Chapter~5]{Stoll}), which reproduces any function $f$ \Mh on $\Bn$ and continuous
on $\overline{\Bn}$ from its boundary values:
\[ f(x) = \int_\pBn f(\zeta) P(x,\zeta) \,d\sigma(\zeta), \qquad\forall x\in\Bn,  \label{tTG}  \]
is~just the boundary value of $K\Sz(x,y)$ as $y\to\zeta$; that~is, $K\Sz(x,\cdot)$ is just
the \Mh extension of $P(x,\cdot)$ from $\pBn$ into~$\Bn$,~or
\[ K\Sz(x,y) = \frac{\Gamma(n)^2}{4\pi^{2n}} \int_\pBn \frac{(1-|x|^2)^n(1-|y|^2)^n}
 {|1-\spr{x,\zeta}|^{2n}|1-\spr{y,\zeta}|^{2n}} \,d\sigma(\zeta).  \label{tTH}  \]
For $n=1$, this is easily evaluated~to
\[ K\Sz(x,y) = \frac1{2\pi} \frac{1-|x|^2|y|^2}{|1-\spr{x,y}|^2},  \label{tTI}  \]
however, again, nothing seems to be known for $n\ge2$.

The~aim of this paper~is, firstly, to~give an explicit formula for the \Mh Szeg\"o kernel
$K\Sz(x,y)$ for any~$n$, in~terms of certain hypergeometric functions; and secondly,
to~give a series expansion for $K_s(x,y)$, for any $n$ and any $s>-1$. This series expansion
is sufficient to show that, on~the one hand, there is probably no explicit formula for
$K_s$ when $n\ge2$ (even for $s=0$, i.e.~in the case of the unweighted \Mh Bergman space);
and, on~the other hand, to~give at least a rough idea of what is the singularity of $K_s(x,y)$
when both $x$ and $y$ approach the boundary~$\pBn$.

To~describe our results, we~recall some facts about hypergeometric functions. The~ordinary
(Gauss) hypergeometric function ${}_2\!F_1$ of one variable is defined~by
$$ \FF21{a,b}cz = \sum_{j=0}^\infty \frac{(a)_j(b)_j}{(c)_j} \, \frac{z^j}{j!},
 \qquad |z|<1 .  $$
Here $c\notin\{0,-1,-2,\dots\}$ while $a,b$ can be any complex numbers, and
$$ (a)_j :=a(a+1)\dots(a+j-1) = \frac{\Gamma(a+j)}{\Gamma(a)}  $$
stands for the Pochhammer symbol (raising factorial). We~have also already met in~\eqref{tTD}
the Appell function $F_1$ of two variables, defined~by
$$ \FE1{a;b_1,b_2}c{x,y} = \sum_{j,k=0}^\infty \frac{(a)_{j+k}(b_1)_j(b_2)_k}{(c)_{j+k}}
 \frac{x^j}{j!} \frac{y^k}{k!}, \qquad |x|<1, |y|<1.  $$
The~hypergeometric function $FD_1$ of four variables
\[ \begin{aligned}
& \FD{a,a',b_1,b_2}c{x_1,x_2,y_1,y_2} = \\
&\hskip2em \sum_{i_1,i_2,j_1,j_2=0}^\infty
 \frac{(a)_{i_1+i_2}(a')_{j_1+j_2}(b_1)_{i_1+j_1}(b_2)_{i_2+j_2}}{(c)_{i_1+i_2+j_1+j_2}}
 \frac{x_1^{i_1}}{i_1!} \frac{x_2^{i_2}}{i_2!} \frac{y_1^{j_1}}{j_1!} \frac{y_2^{j_2}}{j_2!}
 \end{aligned} \label{tTJ}  \]
has been denoted $K_{16}(a,b_1,b_2,a';c;x_1,x_2,y_1,y_2)$ in Exton~\cite[p.~78]{ExtB} and
$K^*_{16}(a,b_1,$ $a',b_2;c;x_2,x_1,y_1,y_2)$ in~\cite{ExtP}; our notation follows Karlsson~\cite{Karls}.
The~series \eqref{tTJ} converges for $x_1,x_2,y_1,y_2\in\DD$.

Our~main results are the following.

\begin{theorem*}[Theorem~\ref{PB}] For any $\alpha,\beta,\gamma,\delta\in\CC$ and $n\ge1$,
\begin{align*}
& \int_\pBn (1-\spr{z,\zeta})^{-\alpha} (1-\spr{\zeta,z})^{-\beta}
 (1-\spr{w,\zeta})^{-\gamma} (1-\spr{\zeta,w})^{-\delta} \,d\sigma(\zeta)  \\
&\hskip4em = \frac{2\pi^n}{\Gamma(n)} \FD{\beta,\delta,\alpha,\gamma}n{|z|^2,\spr{z,w},\spr{w,z},|w|^2} .
\end{align*}
\end{theorem*}

\begin{corollary*}[Corollary~\ref{PC}] For any $n\ge1$,
\[ K\Sz(z,w) = \frac{\Gamma(n)}{2\pi^n} (1-|z|^2)^n(1-|w|^2)^n
 \FD{n,n,n,n}n{|z|^2,\spr{z,w},\spr{w,z},|w|^2} . \label{tTS} \]
\end{corollary*}

The~last right-hand side can be expressed in terms of ordinary ${}_2\!F_1$ functions.

\begin{theorem*}[Theorem~\ref{PD}] For any $n\ge1$, $K\Sz(z,w)$ equals
\begin{align*}
& \frac{\Gamma(n)}{2\pi^n} \frac{(1-|w|^2)^n}{|1-\spr{z,w}|^{2n}} \sum_{i_1=0}^n \sum_{i_2,j_1=0}^{n-i_1}
 \frac{(-n)_{i_1+i_2} (-n)_{i_1+j_1} (n)_{i_2} (n)_{j_1}} {i_1!i_2!j_1!(n)_{i_1+i_2+j_1}} \\
&\hskip6em \;\times t_1^{i_1} t_2^{i_2} t_3^{j_1} \FF21{i_2+n,j_1+n}{i_1+i_2+j_1+n}{t_4},
\end{align*}
where
$$ t_1=|z|^2, \quad t_2=\frac{|z|^2-\spr{w,z}}{1-\spr{w,z}}, \quad t_3=\overline{t_2},
\quad t_4=1-\frac{(1-|z|^2)(1-|w|^2)}{|1-\spr{z,w}|^2} .  $$
\end{theorem*}

Note that for $z,w\in\Bn$ we have $t_2,t_3\in\DD$, while $0\le t_1,t_4<1$.
Note also that using standard formulas for hypergeometric functions (cf.~\eqref{tWU} in Section~\ref{sec3}),
the last ${}_2\!F_1$ are actually expressible in the form $a(t_4)+b(t_4)\log(1-t_4)$
with rational functions $a,b$.

Both the holomorphic Szeg\"o kernel \eqref{tTL} and the harmonic Szeg\"o kernel \eqref{tTF}
are clearly smooth functions on the closure $\overline{\Bn\times\Bn}$ except for the boundary
diagonal $\operatorname{diag}\pBn=\{(x,y)\in\pBn\times\pBn: x=y\}$. This is no longer the case
for the \Mh Szeg\"o kernel.

\begin{proposition*}[Proposition~\ref{PE}] For $n>1$,
$$ K\Sz\in C^{n-1}(\bbd) \setminus C^n(\bbd).  $$
\end{proposition*}

As~for the \Mh kernels $K_s$, $s>-1$, we~may actually consider more general measures
$d\mu\otimes d\sigma$ on $\Bn$ given~by
\[ \int_\Bn F(z) \,(d\mu\otimes d\sigma)(z) := \int_0^1 \int_\pBn F(\sqrt t \zeta)
 \, t^{n-1}\,d\mu(t) \,d\sigma(\zeta),  \label{tTK}  \]
where $d\mu$ is any finite Borel measure on the interval $[0,1]$ such that $1\in\operatorname{supp}d\mu$.
Denote by $L^2\mh(\Bn,d\mu\otimes d\sigma)$ the corresponding \Mh weighted Bergman space
and let $K_\mu$ be its reproducing kernel. The~spaces \eqref{tTB} and their kernels $K_s(x,y)$
thus correspond to the choice $d\mu(t)=\tfrac12(1-t)^s\,dt$.

\begin{theorem*}[Theorem~\ref{PF}] For any $n\ge1$ and $\mu$ as above, $K_\mu$ is given~by
\[ K_\mu(z,w) = \frac{\Gamma(n)}{2\pi^n} (1-|z|^2)^n (1-|w|^2)^n
 \sum_{p,q,j,m=0}^\infty A_{pqjm}(\mu)
 \frac{\spr{z,w}^p \spr{w,z}^q |z|^{2j} |w|^{2m}}{p!q!j!m!} ,  \label{tTN}  \]
where
\[ \begin{aligned}
A_{pqjm}(\mu) &:= \sum_{l=0}^{\min(m,j)}
 \frac{\Gamma(n+p+j)\Gamma(n+q+j)} {\Gamma(n)\Gamma(n+p+q+j+l)}
 \frac{\Gamma(n+p+m)\Gamma(n+q+m)} {\Gamma(n)\Gamma(n+p+q+m+l)} \\
&\hskip2em  \frac{(-1)^l \Gamma(n+p+q+l-1)(n+p+q+2l-1)(-j)_l(-m)_l}
   {\Gamma(n)l! c_{p+l,q+l}(\mu)}, \end{aligned}  \label{tTO}  \]
with
\[ c_{pq}(\mu) := \frac{\Gamma(p+n)^2\Gamma(q+n)^2} {\Gamma(n)^2\Gamma(p+q+n)^2}
 \int_0^1 t^{p+q+n-1} \FF21{p,q}{p+q+n}t ^2 \,d\mu(t).   \label{tTM}  \]
\end{theorem*}

We~remark that for $\mu$ the unit mass at the point $t=1$, $L^2\mh(\Bn,d\mu\otimes d\sigma)$
reduces just to $H^2\mh(\Bn)$ and $K_\mu$ to~$K\Sz$, while one can then show that $c_{pq}(\mu)=1$
for all $p,q$ and $A_{pqjm}=(n)_{j+p}(n)_{j+q}(n)_{m+p}(n)_{m+q}/(n)_{m+j+p+q}$; thus the last
theorem recovers Corollary~\ref{PC} as its special case.

Denoting $c_{pq}(\mu)$ for $d\mu(t)=\tfrac12(1-t)^s\,dt$ by $c_{pq}(s)$, the last theorem
thus gives also a series expansion for the kernels~$K_s$, $s>-1$. We~will show that even
for $s=0$ and $n=2$ (i.e.~the unweighted \Mh Bergman space on~$\BB^2$), $c_{11}(s)/c_{00}(s)$ and,
hence, also $A_{0000}(s)/A_{1100}(s)$, is~of the form $a+b\zeta(3)$, where $a,b$ are nonzero
rational numbers and $\zeta$ stands for the Riemann zeta function. That~is, the Taylor coefficients
of $K_0$ on $\BB^2$ involve $\zeta(3)$ in a nontrivial way. This makes it pretty unlikely that~$K_0$,
and \emph{a fortiori} $K_s$ for general $s>-1$ and $n\ge2$, be~given by any ``nice'' explicit
formula in terms of e.g.~hypergeometric and similar functions.

The~coefficients $c_{pq}(s)$ can be expressed as certain multivariable hypergeometric functions
at unit argument; unfortunately, again these expressions do not seem to lend themselves to an
explicit evaluation. However, we~can at least describe the asymptotic behavior of $c_{pq}(s)$
for large $p,q$, which turns out to be sufficient for getting some idea about the boundary
behavior of the kernels~$K_s$.

\begin{theorem*}[Theorem~\ref{PG}] Let $p,q>0$ be fixed. Then as $\lambda\to+\infty$,
we~have the asymptotic expansion
$$ c_{\lambda p,\lambda q}(s) \approx \frac{\Gamma(2n+s+1)\Gamma(n+s+1)^2\Gamma(s+1)}
 {\Gamma(n)^2\Gamma(2n+2s+2)}
 \frac{\lambda^{-2s-2}}{(pq)^{s+1}} \sum_{j=0}^\infty \frac{a_j(p,q)}{\lambda^j}, $$
where $a_0(p,q)=1$.
\end{theorem*}

For the omitted case $pq=0$, we get directly from \eqref{tTM}
\[ c_{0q}(s) = c_{q0}(s) = \frac{\Gamma(q+n)\Gamma(s+1)}{\Gamma(q+n+s+1)}
 \approx \frac{\Gamma(n)\Gamma(s+1)}{\Gamma(n+s+1)} q^{-s-1}
 \qquad \text{as $q\to+\infty$} . \label{tTR}  \]
Thus in all cases $c_{pq}(s)\approx \frac{\Gamma(n)\Gamma(s+1)}{\Gamma(n+s+1)}
(p+1)^{-s-1}(q+1)^{-s-1}$ as $p+q\to+\infty$. Our~final result, though falling short of
showing the situation for $K_s$ itself, thus at least describes the boundary behavior of
a series whose ``leading order term'' is the same as for~$K_s$.

\begin{theorem*}[Theorem~\ref{PH}] For $n>1$ and $s=0,1,2,\dots$, consider the function
$F_s(z,w)$ given by \eqref{tTN}, \eqref{tTO} but with $c_{pq}(\mu)$ replaced by
$(p+\tfrac{n-1}2)^{-s-1}(q+\tfrac{n-1}2)^{-s-1}$. Then
\[ \begin{aligned}
F_s(z,w) &= \LL^{s+1} \Big[ \frac{\Gamma(n)}{2\pi^n} \frac{(1-y_2)^n}{(1-x_2)^n(1-y_1)^n}
 \sum_{i_1=0}^n \sum_{i_2,j_1=0}^{n-i_1}
 \frac{(-n)_{i_1+i_2} (-n)_{i_1+j_1} (n)_{i_2} (n)_{j_1}} {i_1!i_2!j_1!(n)_{i_1+i_2+j_1}} \\
&\hskip2em x_1^{i_1} \Big(\frac{x_1-y_1}{1-y_1}\Big)^{i_2} \Big(\frac{x_1-x_2}{1-x_2}\Big)^{j_1}
 \FF21{i_2+n,j_1+n}{i_1+i_2+j_1+n}{1-\frac{(1-x_1)(1-y_2)}{(1-x_2)(1-y_1)}} \Big] \\
&\hskip2em  \Big|_{x_1=|z|^2, x_2=\spr{z,w}, y_1=\spr{w,z}, y_2=|w|^2}, \end{aligned} \label{tTP}  \]
where $\LL$ is the linear differential operator
$$ \LL:=(x_2y_1-x_1y_2)\frac{\partial^2}{\partial x_2\partial y_1}
 +\frac{n-1}2 \Big(x_2\frac\partial{\partial x_2}+y_1\frac\partial{\partial y_1}\Big)
 +\frac{(n-1)^2}4 I.  $$
\end{theorem*}

Note that $\LL$ involves differentiations only with respect to $x_2$ and~$y_1$.

Recall that for each $z\in\Bn$, $z\neq0$, there is a (unique) biholomorphic self-map
$\phi_z\in\operatorname{Aut}(\Bn)$ which interchanges $z$ and the origin~$0$; explicitly,
$$ \phi_z(w) = \frac{z-P_z w - \sqrt{1-|z|^2}(w-P_z w)}{1-\spr{w,z}},
 \qquad P_z w:=\frac{\spr{w,z}}{|z|^2}z.  $$
For $z=0$, we set $\phi_0(w):=-w$. One~has the useful formula \cite[Theorem~2.2.2]{Ru}
\[ 1-\spr{\phi_z w_1,\phi_z w_2} = \frac{(1-|z|^2)(1-\spr{w_1,w_2})}
 {(1-\spr{z,w_2})(1-\spr{w_1,z})}, \label{tTT}  \]
from which it follows that the various quantities appearing in \eqref{tTP} are in fact
given~by
\begin{align*}
& x_1=|z|^2, \quad \frac{x_1-x_2}{1-x_2}=\spr{z,\phi_z w}, \qquad
 \frac{x_1-y_1}{1-y_1}=\spr{\phi_z w,z}, \quad \\
&\hskip2em  1-\frac{(1-x_1)(1-y_2)}{(1-x_2)(1-y_1)} = |\phi_z w|^2.
\end{align*}
We~set
\[ U:=|z|^2, \quad V:=|\phi_z w|^2, \quad Z:=\spr{z,\phi_z w}; \label{tTQ}  \]
it~will be shown that the map $(z,w)\mapsto(U,V,Z)$ is actually a bijection~of
$$ \text{$\Bn\times\Bn$ modulo the action of the group $\Un$ of unitary maps of $\CC^n$} $$
onto the set
$$ \Omega:=\{(U,V,Z): 0\le U,V<1, Z\in\CC, |Z|^2\le UV \}.  $$
The~coordinates $U,V,Z$ are well suited for the description of the singularity of $F_s$
near the boundary diagonal.

\begin{corollary*}[Corollaries~\ref{PI} and~\ref{PJ}] For $n>1$ and $s=0,1,2,\dots$,
$$ F_s\in C^{n-1}(\bbd), $$
and
\begin{align*}
 F_s(z,w) &= \frac{(1-V)^n}{(1-U)^{n+s+1}}
 \sum_{i_1=0}^n \sum_{i_2,j_1=0}^{n-i_1} \sum_{k=0}^{s+1}
 P_{i_1 i_2 j_1 k}(U,Z,\oZ,V) \\
&\hskip2em \;\times \FF21{i_2+m+k,j_1+n+k}{i_1+i_2+j_1+n+k}V,
\end{align*}
where $P_{i_1 i_2 j_1 k}(U,Z,\oZ,V)$ is a polynomial of degree at most
$i_1+s+1$, $j_1+s+1$, $i_2+s+1$ and $k+s+1$, respectively, in~the indicated variables.
\end{corollary*}

We~expect the boundary behavior of the \Mh kernels $K_s$ to be of the same nature as for~$F_s$.

The~paper is organized as follows. In~Section~\ref{sec2}, we~review the necessary prerequisites
on the Peter-Weyl decomposition of \Mh functions under the action of the unitary group $\Un$
of~$\CC^n$. The~results about the \Mh Szeg\"o kernel are proved in Section~\ref{sec3},
and those about $K_\mu$ in Section~\ref{sec4}. The~asymptotic expansion of $c_{pq}(s)$ and
the assertions about $F_s$ are derived in Section~\ref{sec5}. Some final remarks, comments and
open problems are collected in the final section, Section~\ref{sec6}.

To~make typesetting a little neater, the shorthand
$$ \GG{a_1,a_2,\dots,a_k}{b_1,b_2,\dots,b_m} :=
 \frac{\Gamma(a_1)\Gamma(a_2)\dots\Gamma(a_k)}
  {\Gamma(b_1)\Gamma(b_2)\dots\Gamma(b_m)}  $$
is often employed throughout the paper. The~inner product $\spr{z,w}$ of $z,w\in\CC^n$
is sometimes also written as $\cdd zw$; and $\phi_z(w)$ is frequently abbreviated
just to $\phi_z w$ (which we actually already did a few paragraphs above).

\section{Notation and preliminaries} \label{sec2}
The stabilizer of the origin $0\in\Bn$ in $\operatorname{Aut}(\Bn)$ is the group $\Un$
of all unitary transformations of~$\CC^n$; that~is, of~all linear operators $U$ that
preserve inner products:
$$ \spr{Uz,Uw} = \spr{z,w} \qquad \forall z,w\in\CC^n.   $$
Each $U\in\Un$ maps the unit sphere $\pBn$ onto itself, and the surface measure
$d\sigma$ on $\pBn$ is invariant under~$U$. It~follows that the composition with
elements of~$\Un$,
\[ T_U: f \mapsto f \circ U^{-1},  \label{tVA}  \]
is~a unitary representation of $\Un$ on $L^2(\pBn,d\sigma)$. We~will need the decomposition
of this representation into irreducible subspaces. These turn out to be given by
\emph{bigraded spherical harmonics}~$\chpq$; the standard sources for this are
Rudin~\cite[Chapter 12.1--12.2]{Ru}, or Krantz~\cite[Chapter 6.6--6.8]{KraPDECA},
with basic ingredients going back to Folland~\cite{Foll}.

Namely, for integers $p,q\ge0$, let $\chpq$ be vector space of restrictions to $\pBn$
of harmonic polynomials $f(z,\oz)$ on $\CC^n$ which are homogeneous of degree $p$ in $z$
and homogeneous of degree $q$ in~$\oz$. Then $\chpq$ is invariant under the action
\eqref{tVA} of~$\Un$, is $\Un$-irreducible (i.e.~has no proper $\Un$-invariant subspace)
and
\[ L^2(\pBn,d\sigma) = \bigoplus_{p,q=0}^\infty \chpq .  \label{tVB}  \]
Furthermore, the representations of $\Un$ on $\chpq$ are mutually inequivalent;
that~is, if $T:\chpq\to\mathcal H^{kl}$ is a linear operator commuting with the
action~\eqref{tVA}, then necessarily
\[ \begin{cases} T=0 & \text{ if } (k,l)\neq(p,q), \\
T=cI|_{\chpq} \text{ for some }c\in\CC & \text{ if }(k,l)=(p,q),
\end{cases} \label{tVH}  \]
where $I$ denotes the identity operator.

Since each space $\chpq$ is finite-dimensional, the evaluation functional $f\mapsto f(\zeta)$
at each $\zeta\in\pBn$ is automatically continuous on~it; it~follows that $\chpq$ --- with the
inner product inherited from $L^2(\pBn,d\sigma)$ --- has a reproducing kernel. This reproducing
kernel turns out to be given by $H^{pq}(\cdd\zeta\eta)$, where for $n\ge2$
\[ \begin{aligned}
H^{pq}(re^{i\theta}) &= \frac{(p+q+n-1)(p+n-2)!(q+n-2)!}{p!q!(n-1)!(n-2)!} \\
&\hskip2em \,\times r^{|p-q|} e^{(p-q)i\theta} \frac{\Gamma(n)}{2\pi^n}
 \frac{P^{(n-2,|p-q|)}_{\min(p,q)}(2r^2-1)}{P^{(n-2,|p-q|)}_{\min(p,q)}(1)},
 \end{aligned} \label{tVC}  \]
where
$$ P^{(\alpha,\beta)}_m(x) = (-1)^n \frac{(1-x)^{-\alpha}(1+x)^{-\beta}}{m!2^m}
 \frac{d^m}{dx^m} [(1-x)^{\alpha+m}(1+x)^{\beta+m}]  $$
are the Jacobi polynomials. Thus
$$ f(\zeta) = \int_\pBn f(\eta) H^{pq}(\cdd\zeta\eta)\,d\sigma(\eta),
 \qquad\forall \zeta\in\pBn, \forall f\in\chpq.  $$
In~particular, we~have the orthogonality relations
\[ \int_\pBn H^{pq}(\cdd\zeta\eta) H^{kl}(\cdd\eta\xi) \,d\sigma(\eta)
 = \delta_{pk}\delta_{ql} H^{pq}(\cdd\zeta\xi).   \label{tVD}   \]
Note that by \cite[formula 10.8(16)]{BE},
$$ P^{(\alpha,\beta)}_m(x) = (-1)^m \binom{m+\beta}m \FF21{-m,m+\alpha+\beta+1}{\beta+1}{\frac{1+x}2}, $$
so~we have
\[ \begin{aligned}
H^{pq}(z) &= \frac{\Gamma(n)}{2\pi^n} \frac{(-1)^q(n+p+q-1)(n+p-2)!}{(n-1)!q!(p-q)!}  \\
&\hskip4em \,\times z^{p-q} \FF21{-q,n+p-1}{p-q+1}{|z|^2} \qquad\text{for }p\ge q, \end{aligned} \label{tVE}  \]
while $H^{pq}(z)=H^{qp}(\oz)$ for $p<q$. This formula will be useful later~on.

Denote
\[ \begin{aligned}
S^{pq}(r) &:= r^{p+q} \FF21{p,q}{p+q+n}{r^2} \Big/ \FF21{p,q}{p+q+n}1 \\
 &= \GG{p+n,q+n}{n,p+q+n} r^{p+q} \FF21{p,q}{p+q+n}{r^2} . \end{aligned}
 \label{tVF}  \]
Then for each $f\in\chpq$, the (unique) solution to the Dirichlet problem $\td u=0$ on~$\Bn$,
$u|_\pBn=f$ is given~by
\[ u(r\zeta) = S^{pq}(r) f(\zeta), \qquad 0\le r\le 1, \; \zeta\in\pBn.  \label{tVG}  \]

For $n=1$, all the above remains in force, only the spaces $\chpq$ reduce just to $\{0\}$
if $pq\neq0$, while $\mathcal H^{p0}=\CC z^p$, $\mathcal H^{0q}=\CC\oz^q$ and
$H^{p0}(z)=\frac1{2\pi}z^p$, $H^{0q}(z)=\frac1{2\pi}\oz^q$; note that the formula \eqref{tVE} still works for $n=1$ and $pq=0$.

For each fixed $x\in\Bn$, the \Mh Poisson kernel $P(x,\cdot)$ always belongs to $L^2(\pBn,d\sigma)$
(it~is a smooth function on the sphere), hence it can be decomposed into the $\chpq$ components as
in~\eqref{tVB}. This decomposition was obtained by Folland~\cite{Foll}:
\[ P(r\zeta,\eta) = \sum_{p,q=0}^\infty S^{pq}(r) H^{pq}(\cdd\zeta\eta),
 \qquad 0\le r<1, \; \zeta,\eta\in\pBn.  \label{tVI}  \]
Folland gave his proof for $n\ge2$, but with the caveat from the preceding paragraph it actually
holds also for $n=1$. (We~will give an alternative proof for any $n$ in Remark~\ref{reFo} in
Section~\ref{sec4}). The~sum converges pointwise, uniformly for $\eta\in\pBn$ and $r\zeta$ in
a compact subset of~$\Bn$, as~well as in $L^2(\pBn,d\sigma)$ for each fixed $r$ and~$\zeta$.

Using the orthogonality relations~\eqref{tVD}, one~can get from \eqref{tVI} the analogous
decomposition for the \Mh Szeg\"o kernel. Namely, starting from~\eqref{tTH}:
$$ K\Sz(x,y) = \int_\pBn P(x,\zeta) \overline{P(y,\zeta)} \,d\sigma(\zeta)  $$
(note that the complex conjugation actually has no effect, since $P(y,\zeta)$ is real-valued),
and substituting \eqref{tVI} for $P(x,\zeta)$ and $P(y,\zeta)$, we~get
\begin{align}
K\Sz(r\zeta,R\xi) &= \sum_{p,q,k,l=0}^\infty S^{pq}(r) S^{kl}(R)
 \int_\pBn H^{pq}(\cdd\zeta\eta) H^{kl}(\cdd\eta\xi) \,d\sigma(\eta)  \nonumber \\
&= \sum_{p,q=0}^\infty S^{pq}(r) S^{pq}(R) H^{pq}(\cdd\zeta\xi)
 \qquad\text{by \eqref{tVD}}, \label{tVJ}   \end{align}
the interchange of integration and summation being justified by the $L^2$-convergence.

We~conclude this section by giving a similar formula for the reproducing kernel of any
Hilbert space of \Mh functions on $\Bn$ with a $\Un$-invariant inner product.

For each $p,q\ge0$, let $\bhpq$ be he space of all functions on $\Bn$ of the form
\eqref{tVG} with $f\in\chpq$. In~other words, while $\chpq$ is the space of spherical
harmonics on the sphere~$\pBn$, $\bhpq$ is the associated space of ``solid'' \Mh functions
on~$\Bn$. With the inner product inherited from~$L^2(\pBn,d\sigma)$, each $\bhpq$ is thus
a finite-dimensional Hilbert space of \Mh functions on~$\Bn$, unitarily isomorphic to
the space $\chpq$ via the isomorphism~\eqref{tVG}, and with reproducing kernel
\[ K^{pq}(r\zeta,R\xi) := S^{pq}(r) S^{pq}(R) H^{pq}(\cdd\zeta\xi). \label{tVM}  \]

\begin{proposition} \label{PA}
Let $\HH$ be any Hilbert space of \Mh functions on $\Bn$ which contains $\bhpq$ for
all $p,q\ge0$, is~invariant under the action~\eqref{tVA} (i.e.~$T_U f\in\HH$ whenever
$f\in\HH$ and $U\in\Un$) and whose inner product is $\Un$-invariant:
\[ \spr{T_U f,T_U g}_\HH = \spr{f,g}_\HH \qquad \forall f,g\in\HH, U\in\Un.  \label{tVK}  \]
Then
\begin{itemize}
\item[(i)] the spaces $\bhpq$ are pairwise orthogonal in~$\HH$;
\item[(ii)] on~each $\bhpq$, the $\HH$-inner product is a constant multiple of the
$L^2(d\sigma)$-inner product: there exist finite constants $c_{pq}>0$ such that
\[ \spr{f,g}_\HH = c_{pq} \spr{f,g}_{L^2(\pBn,d\sigma)} \qquad\forall f,g\in\bhpq .
 \label{tVN} \]
\end{itemize}
Furthermore, if~the action $U\mapsto T_U$ of $\Un$ on $\HH$ is strongly continuous
(i.e.~for each $f\in\HH$, $U\mapsto T_U f$ is continuous from $\Un$ into~$\HH$),
then additionally
\begin{itemize}
\item[(iii)] the linear span of $\bhpq$, $p,q\ge0$, is~dense in~$\HH$;
\item[(iv)] if~the point evaluations are bounded on~$\HH$, then the reproducing kernel
$K_\HH$ of $\HH$ is given~by
\[ K_\HH(r\zeta,R\xi) = \sum_{p,q=0}^\infty \frac{S^{pq}(r) S^{pq}(R)
 H^{pq}(\cdd\zeta\xi)}{c_{pq}},  \label{tVO}  \]
with the sum converging pointwise and locally uniformly on compact subsets of $\Bn\times\Bn$,
as~well as in~$\HH$ as a function of $x=r\zeta$ for each fixed $y=R\xi$, or~vice versa.
\end{itemize}
\end{proposition}

Note that the last proposition applies, in~particular, to $\HH=L^2(\pBn,d\sigma)$;
in~that case trivially $c_{pq}=1$ $\forall p,q$, so~we recover~\eqref{tVJ}.

\begin{proof} The~restriction of the inner product in $\HH$ to $\bhpq\times\mathbf H^{kl}$
is a continuous sesqui-linear form on these (finite dimensional) spaces, so~by the
Riesz-Fischer theorem
$$ \spr{f,g}_\HH = \spr{Tf,g}_{L^2(d\sigma)} \qquad \forall f\in\bhpq, g\in\mathbf H^{kl} $$
for some linear operator $T:\bhpq\to\mathbf H^{kl}$. By~\eqref{tVK}, $T$~is $\Un$-invariant;
thus by~\eqref{tVH}, $T=0$ if $(p,q)\neq(k,l)$ while $T=c_{pq}I$ if $(p,q)=(k,l)$.
This proves (i) and~(ii).

To prove (iii), let $f\in \HH$. Since $f$ is $M$-harmonic, by Theorem~2.1 of \cite{ABC},
$f$~can be expanded in the form
\[ f = \sum_{p,q} f_{pq}  \label{AUX} \]
where $f_{pq} \in \mathbf H^{pq}$ for all $p,q$, and the series converges uniformly on
compact subsets of $\Bn$. Furthermore, $f_{pq}$ is actually given explicitly~by
$$ f_{pq}(r\zeta) = \int_\pBn f(r\eta) H^{pq}(\cdd\zeta\eta)\,d\sigma(\eta)  $$
(see \cite[p.~107]{ABC}). Setting $\eta=U^{-1}\zeta$, this can also be written~as
\[ f_{pq} = \int_\Un \chi^{pq}(U) \, T_U f \,dU ,  \label{AUY} \]
where $dU$ is the Haar measure on the compact group $\Un$, normalized to be of total mass
$\frac{2\pi^n}{\Gamma(n)}$, and $\chi^{pq}(U):=H^{pq}(\cdd{U\zeta}\zeta)$
(this function --- the~character of the representation $T_U$ on~$\chpq$ --- does not depend on~$\zeta$).
Now by the hypothesis of strong continuity of~$T_U$, the~last integral exists also
as a Bochner integral, i.e.~converges also in~$\HH$. Also, by~an elementary estimate of the
same integral, the map $P_{pq}:f\mapsto f_{pq}$ is continuous from $\HH$ into~$\bhpq\subset\HH$.
Making the change of variable $U\mapsto U^{-1}$ in \eqref{AUY} shows that $P_{pq}$ is self-adjoint,
and since $P_{pq}$ clearly reduces to the identity on $\bhpq$, it~follows that $P_{pq}$ has to be the
precisely the projection in $\HH$ onto~$\bhpq\subset\HH$.

Now if $f\in\HH$ is orthogonal to all $\bhpq$, then $f_{pq}=P_{pq}f=0$ for all $p,q\ge0$;
thus by \eqref{AUX} $f=0$. It~follows that the linear span of $\bhpq$, $p,q\ge0$, is dense in~$\HH$,
proving~(iii).

As~for (iv), recall that for any functional Hilbert space with bounded point evaluations,
the~reproducing kernel is given by the formula
\[ K_\HH(x,y) = \sum_j f_j(x) \overline{f_j(y)}   \label{tVL}   \]
where $\{f_j\}_j$ is any orthonormal basis of~$\HH$; see~\cite{Aron}. In~our case, thanks to
(i)--(ii), we~can choose an orthonormal basis of the form $\{f_{pqj}/\sqrt{c_{pq}}\}_{pqj}$,
where for each $p,q\ge0$, $\{f_{pqj}\}_{j=1}^{\dim\bhpq}$ is an orthonormal basis in $\bhpq$
with respect to the $L^2(d\sigma)$ inner product. Thus
$$ K_\HH(x,y) = \sum_{p,q=0}^\infty \frac1{c_{pq}}
 \sum_{j=1}^{\dim\bhpq} f_{pqj}(x)\overline{f_{pqj}(y)}. $$
However, by~\eqref{tVL} now applied to $\bhpq$ with the $L^2(d\sigma)$ inner product,
the~inner sum is precisely the reproducing kernel of $\bhpq$ with respect to the $L^2(d\sigma)$
inner product, which we know to be given by~\eqref{tVM}. This settles~\eqref{tVO}. The~claim
concerning convergence in $\HH$ is immediate from the same property for~\eqref{tVL}
(cf.~again~\cite{Aron}), while for the uniform convergence on compact subsets of $\Bn\times\Bn$
it~is, similarly, enough to show that the norms $\|K_\HH(\cdot,z)\|_\HH=K_\HH(z,z)^{1/2}$
stay bounded if $z$ ranges in a compact subset of~$\Bn$. However, this is immediate from the
fact that $K_\HH(x,y)$, being \Mh in each variable, is~real-analytic in $(x,y)\in\Bn\times\Bn$
by the standard elliptic regularity theory; in~particular, $K_\HH(z,z)$ is a continuous
function on~$\Bn$. This completes the proof.   %\qed
\end{proof}

% Plainly, \eqref{tVJ} is the special case of \eqref{tVO} corresponding to $c_{pq}=1$ $\forall p,q$.

We~remark the Proposition~\ref{PA} remains in force even when the hypothesis that $\bhpq\subset\HH$
for all $p,q\ge0$ is dropped. Indeed, denoting in that case
$$ Y^{pq} := \HH \cap \bhpq,  $$
it~follows from the $\Un$-irreducibility of $\chpq$ (and, hence, of~$\bhpq$) that if $Y^{pq}\neq\{0\}$,
then already $Y^{pq}=\bhpq$ (and, so, $\bhpq\subset\HH$). All~the items (i)--(iv) then remain in force,
except for the fact that in (iii) the linear span of $Y^{pq}$ is dense in $\HH$ and in (iv) instead of all
$p,q\ge0$ one takes only those $p,q$ for which $Y^{pq}\neq\{0\}$. We~are leaving the details to the reader.

\section{The \Mh Szeg\"o kernel} \label{sec3}
\begin{theorem}  \label{PB}
For any $\alpha,\beta,\gamma,\delta\in\CC$ and $n\ge1$,
\[ \begin{aligned}
& \int_\pBn (1-\spr{z,\zeta})^{-\alpha} (1-\spr{\zeta,z})^{-\beta}
 (1-\spr{w,\zeta})^{-\gamma} (1-\spr{\zeta,w})^{-\delta} \,d\sigma(\zeta) \\
&\hskip4em \,= \frac{2\pi^n}{\Gamma(n)} \FD{\beta,\delta,\alpha,\gamma}n{|z|^2,\spr{z,w},\spr{w,z},|w|^2} .
\end{aligned} \label{tWA}  \]
\end{theorem}

Recall that the function $FD_1$ has been defined in~\eqref{tTJ}.

\begin{proof} Clearly the integrand in~\eqref{tWA} remains unchanged if $z,w,\zeta$ are replaced
by $Uz,Uw,U\zeta$, respectively, with any $U\in\Un$. Since the surface measure $d\sigma$ on $\pBn$
is $\Un$-invariant, it~therefore follows that the integral \eqref{tWA} remains unchanged if $z,w$
are replaced by $Uz,Uw$, for any $U\in\Un$. Now~we can pick $U\in\Un$ which maps $z$ into $|z|e_1$,
where $e_1=(1,0,\dots,0)\in\CC^n$. This $U$ sends $w$ into a point in $\Bn$ whose first coordinate
--- denote it by~$b$ --- satisfies $b|z|=\spr{z,w}$; for $n>1$, we~can then continue by choosing
a suitable element of $U(n-1)$, acting on the remaining $n-1$ coordinates, so~that in the end $w$
is mapped into the point $be_1+ce_2$, $e_2=(0,1,0,\dots,0)$, where $c=\sqrt{|w|^2-|b|^2}$.
Altogether, we~thus see that for $n>1$, the integral \eqref{tWA} is equal~to
\[ \int_\pBn (1-a\ozeta_1)^{-\alpha} (1-\oa\zeta_1)^{-\beta} (1-b\ozeta_1-c\ozeta_2)^{-\gamma}
 (1-\ob\zeta_1-\oc\zeta_2)^{-\delta} \,d\sigma(\zeta),  \label{tWB}  \]
where
\[ \begin{cases}
 a=|z|, \; b=\frac{\spr{z,w}}{|z|}, \; c=\sqrt{|w|^2-|b|^2} \qquad & \text{for }z\neq0, \\
 a=b=0, \; c=\sqrt{|w|^2-|b|^2} (=|w|) & \text{for }z=0. \end{cases}  \label{tWC}  \]
For $n=1$, this has to be replaced~by
\[ \begin{cases}
 a=|z|, \; b=\frac{\spr{z,w}}{|z|} \qquad & \text{for }z\neq0,  \\
 a=0, \; b=|w| & \text{for } z=0,  \end{cases}  \label{tWD}  \]
while $c=0$ (so~that $c\ozeta_2$ and $\oc\zeta_2$ in \eqref{tWB} both disappear).

Let~us now compute the integral~\eqref{tWB}. The~binomial expansion
\[ (1-z)^{-\nu} = \sum_{j=0}^\infty \frac{(\nu)_j}{j!} z^j, \qquad\nu\in\CC,   \label{tWI} \]
converges uniformly for $z$ in a compact subset of~$\DD$. Substituting this into \eqref{tWD}
four times, we~see that \eqref{tWB} equals
\begin{align*}
& \sum_{j,k,l,m=0}^\infty \frac{(\alpha)_j(\beta)_k(\gamma)_l(\delta)_m} {j!k!l!m!}
 \int_\pBn (a\ozeta_1)^j(\oa\zeta_1)^k(b\ozeta_1+c\ozeta_2)^l(\ob\zeta_1+\oc\zeta_2)^m \,d\sigma(\zeta) \\
&= \sum_{j,k,l,m=0}^\infty \frac{(\alpha)_j(\beta)_k(\gamma)_l(\delta)_m} {j!k!l!m!}
 \int_\pBn (a\ozeta_1)^j(\oa\zeta_1)^k \sum_{p=0}^l \sum_{q=0}^m \\
&\hskip4em \binom lp \binom mq (b\ozeta_1)^p(c\ozeta_2)^{l-p} (\ob\zeta_1)^q(\oc\zeta_2)^{m-q}.
\end{align*}
Using the familiar formula, valid for any multiindices $\nu,\mu$,
\[ \int_\pBn \zeta^\nu \ozeta^\mu \,d\sigma(\zeta) = \delta_{\nu\mu}
 \frac{\nu!}{(n)_{|\nu|}} \frac{2\pi^n}{\Gamma(n)} ,   \label{tWE}   \]
it~transpires that \eqref{tWB} equals
\[ \begin{aligned}
& \frac{2\pi^n}{\Gamma(n)} \sum_{j,k,l,m=0}^\infty \frac{(\alpha)_j(\beta)_k(\gamma)_l(\delta)_m}{j!k!l!m!}
 \int_\pBn (a\ozeta_1)^j(\oa\zeta_1)^k
 \sum_{p=0}^l \sum_{q=0}^m \binom lp \binom mq \\
&\hskip6em \,\times a^j \oa^k b^p \ob^q c^{l-p} \oc^{m-q}
   \delta_{p-q,l-m}\delta_{p-q,k-j} \frac{(k+q)!(m-q)!}{(n)_{k+m}}. \end{aligned} \label{tWF}  \]
We~first deal with the summands for which $p\ge q$ --- say, $p=q+r$ with some $r\ge0$.
The~two delta functions are then nonzero only if $l=m+r$ and $k=j+r$.
Thus the sum of all such summands will equal
\[ \sum_{r=0}^\infty \oa^r b^r \sum_{j,m=0}^\infty
 \frac{(\alpha)_j(\beta)_{j+r}(\gamma)_{m+r}(\delta)_m} {j!(j+r)!(m+r)!m!}
  \sum_{q=0}^m \binom{m+r}{q+r} \binom mq |a|^{2j} |b|^{2q} |c|^{2m-2q}
  \frac{(j+q+r)!(m-q)!}{(n)_{j+r+m}} .  \label{tWG}  \]
Since $\binom{m+r}{q+r}\binom mq=\frac{(m+r)!m!}{q!(q+r)!(m-q)!^2}$, we~can continue~by
$$ \sum_{r=0}^\infty \oa^r b^r \sum_{j,m=0}^\infty
 \frac{(\alpha)_j(\beta)_{j+r}(\gamma)_{m+r}(\delta)_m} {j!(j+r)!}
  \sum_{q=0}^m |a|^{2j} |b|^{2q} |c|^{2m-2q} \frac{(j+q+r)!}{(n)_{j+r+m}q!(q+r)!(m-q)!} . $$
Writing $m=q+k$, this becomes
$$ \sum_{r=0}^\infty \oa^r b^r \sum_{j,q,k=0}^\infty
 \frac{(\alpha)_j(\beta)_{j+r}(\gamma)_{q+k+r}(\delta)_{q+k}} {j!(j+r)!}
  |a|^{2j} |b|^{2q} |c|^{2k} \frac{(j+q+r)!}{(n)_{j+r+q+k}q!(q+r)!k!} .  $$
Substituting $|c|^{2k}=\sum_{l=0}^k \binom kl (-1)^l |b|^{2l} |w|^{2k-2l}$ from~\eqref{tWC},
this takes the form
$$ \sum_{r=0}^\infty \oa^r b^r \sum_{j,q,k=0}^\infty
 \frac{(\alpha)_j(\beta)_{j+r}(\gamma)_{q+k+r}(\delta)_{q+k}} {j!(j+r)!(q+r)!q!}
  \frac{(j+q+r)!}{(n)_{j+r+q+k}} \sum_{l=0}^k
  \frac{(-1)^l}{l!(k-l)!} |a|^{2j} |b|^{2q+2l} |w|^{2k-2l} ,  $$
or, writing $k=l+m$,
$$ \sum_{r=0}^\infty \oa^r b^r \sum_{j,q,l,m=0}^\infty
 \frac{(\alpha)_j(\beta)_{j+r}(\gamma)_{q+l+m+r}(\delta)_{q+l+m}} {j!(j+r)!(q+r)!q!}
  \frac{(j+q+r)!}{(n)_{j+r+q+l+m}}
  \frac{(-1)^l}{l!m!} |a|^{2j} |b|^{2q+2l} |w|^{2m} .  $$
Setting $q+l=k$, this becomes
$$ \sum_{r=0}^\infty \oa^r b^r \sum_{j,m,k=0}^\infty \sum_{q=0}^k
 \frac{(\alpha)_j(\beta)_{j+r}(\gamma)_{k+m+r}(\delta)_{k+m}} {j!(j+r)!(q+r)!q!}
  \frac{(j+q+r)!}{(n)_{j+r+k+m}}
  \frac{(-1)^{k-q}}{(k-q)!m!} |a|^{2j} |b|^{2k} |w|^{2m} ,  $$
or, since $k!/(k-q)!=(-1)^q(-k)_q$,
\begin{align*}
& \sum_{r=0}^\infty \oa^r b^r \sum_{j,m,k=0}^\infty \sum_{q=0}^k
 \frac{(\alpha)_j(\beta)_{j+r}(\gamma)_{k+m+r}(\delta)_{k+m}} {j!(j+r)!(q+r)!q!}
  \frac{(j+q+r)!}{(n)_{j+r+k+m}}
  \frac{(-1)^k(-k)_q}{k!m!} |a|^{2j} |b|^{2k} |w|^{2m}  \\
&= \sum_{r=0}^\infty \oa^r b^r \sum_{j,m,k=0}^\infty
 \frac{(\alpha)_j(\beta)_{j+r}(\gamma)_{k+m+r}(\delta)_{k+m}} {j!(j+r)!k!m!}
  \frac{(-1)^k}{(n)_{j+r+k+m}} |a|^{2j} |b|^{2k} |w|^{2m}
  \frac{(j+r)!}{r!} \FF21{-k,j+r+1}{r+1}1 .
\end{align*}
By~the Chu-Vandermonde formula,
$$ \frac{(j+r)!}{r!} \FF21{-k,j+r+1}{r+1}1 = \frac{(j+r)!}{r!} \frac{(-j)_k}{(r+1)_k}
 = \frac{(-j)_k(j+r)!}{(r+k)!},  $$
so~we finally get
$$ \sum_{r=0}^\infty \oa^r b^r \sum_{j,m,k=0}^\infty
 \frac{(\alpha)_j(\beta)_{j+r}(\gamma)_{k+m+r}(\delta)_{k+m}} {j!(k+r)!k!m!}
  \frac{(-1)^k(-j)_k}{(n)_{j+r+k+m}} |a|^{2j} |b|^{2k} |w|^{2m} .  $$
Since $(-j)_k$ vanishes for $j<k$, the sum effectively extends only over $j\ge k$,
say, $j=k+l$; as $(-k-l)_k=(-1)^k(l+1)_k=(-1)^k(l+k)!/l!$, we~thus obtain that
\eqref{tWG} is equal~to
$$ \sum_{r=0}^\infty \oa^r b^r \sum_{m,k,l=0}^\infty
 \frac{(\alpha)_{k+l}(\beta)_{k+l+r}(\gamma)_{k+m+r}(\delta)_{k+m}} {l!(k+r)!k!m!(n)_{l+r+2k+m}}
  |a|^{2k+2l} |b|^{2k} |w|^{2m} .  $$
Since by \eqref{tWC} always $a=|z|$ and $\oa b=\spr{z,w}$, we~finally arrive~at
$$ \sum_{k,l,m,r=0}^\infty
 \frac{(\alpha)_{k+l}(\beta)_{k+l+r}(\gamma)_{k+m+r}(\delta)_{k+m}} {l!(k+r)!k!m!(n)_{l+r+2k+m}}
  \spr{z,w}^{k+r} \spr{w,z}^k |z|^{2l} |w|^{2m} ,  $$
or, rechristening $k+r$, $k$ and $l$ to $q+r=p$, $q$ and $j$, respectively,
\[ \sum_{\substack{p,q,j,m=0\\p\ge q}}^\infty
 \frac{(\alpha)_{q+j}(\beta)_{p+j}(\gamma)_{p+m}(\delta)_{q+m}} {p!q!j!m!(n)_{p+q+j+m}}
  \spr{z,w}^p \spr{w,z}^q |z|^{2j} |w|^{2m} .  \label{tWH}  \]
This came from the summands in \eqref{tWF} with $p\ge q$; in~the same way, the sum over $p<q$
in~\eqref{tWF} turns out to be given again by~\eqref{tWH}, but with the summation extending
over $p<q$. Putting these two pieces together, we~thus see that the integral \eqref{tWB} is
equal~to
\begin{align*}
& \frac{2\pi^n}{\Gamma(n)} \sum_{p,q,j,m=0}^\infty
 \frac{(\alpha)_{q+j}(\beta)_{p+j}(\gamma)_{p+m}(\delta)_{q+m}} {p!q!j!m!(n)_{p+q+j+m}}
  \spr{z,w}^p \spr{w,z}^q |z|^{2j} |w|^{2m}  \\
&= \frac{2\pi^n}{\Gamma(n)} \FD{\beta,\delta,\alpha,\gamma}n{|z|^2,\spr{z,w},\spr{w,z},|w|^2},
\end{align*}
as~claimed. This completes the proof for the case $n>1$.

For $n=1$ and $z\neq0$, we~still have $\sqrt{|w|^2-|b|^2}=0=c$ since $|\spr{z,w}|/|z|=|w|$
in this case; thus the whole argument above still works without change. Finally, for $n=1$
and $z=0$, the~integral \eqref{tWB} reduces just~to
$$ \int_{\partial\BB^1} (1-|w|\zeta)^{-\gamma} (1-|w|\ozeta)^{-\delta} \,d\sigma(\zeta)
 = 2\pi \FF21{\gamma,\delta}1{|w|^2}  $$
by~\eqref{tWI}, while
$$ \FD{\beta,\delta,\alpha,\gamma}1{0,0,0,|w|^2} = \FF21{\gamma,\delta}1{|w|^2}  $$
by~\eqref{tTJ}. Thus the assertion holds in this case as well.   %\qed
\end{proof}

\begin{remark} If~we carry out the integration over $(\zeta_3,\dots,\zeta_n)$ in~\eqref{tWB},
the integral transforms into
$$ \frac{2\pi^{n-1}}{\Gamma(n-1)} \int_{\BB^2} (1-a\ozeta_1)^{-\alpha} (1-\oa\zeta_1)^{-\beta}
 (1-b\ozeta_1-c\ozeta_2)^{-\gamma} (1-\ob\zeta_1-\oc\zeta_2)^{-\delta} \,
 (1-|\zeta_1|^2-|\zeta_2|^2)^{n-\frac32} \,d\zeta_1 \,d\zeta_2.  $$
This is strangely reminiscent of the following known integral formula for $FD_1$, valid for
$b_1,b_2>0,$ and $c-b_1-b_2>0$ \cite[formula 4.3.(8)]{Karls}:
\[ \begin{aligned}
& \FD{a,a',b_1,b_2}c{x_1,x_2,y_1,y_2} = \GG{c}{b_1,b_2,c-b_1-b_2} \\
& \hskip2em \times \int_{\substack{u_1,u_2>0\\u_1+u_2<1}}
 \frac{u_1^{b_1-1} u_2^{b_2-1} (1-u_1-u_2)^{c-b_1-b_2-1}}
  {(1-x_1 u_1-x_2 u_2)^a (1-y_1 u_1-y_2 u_2)^{a'}}
   \, du_1\,du_2 . \end{aligned} \label{tWJ}  \]
However, it~does not seem possible to derive our Theorem~\ref{PB} from~\eqref{tWJ}. \qed
\end{remark}

The~following corollary to the last theorem is immediate from~\eqref{tTH}.

\begin{corollary} \label{PC}
For any $n\ge1$,
\[ K\Sz(z,w) = \frac{\Gamma(n)}{2\pi^n} (1-|z|^2)^n(1-|w|^2)^n
 \FD{n,n,n,n}n{|z|^2,\spr{z,w},\spr{w,z},|w|^2} . \label{tWK} \]
\end{corollary}

\begin{remark}
A~posteriori, it~is possible to give a ``direct'' proof of the last corollary by checking
straight away that, for each fixed $w\in\Bn$, the right-hand side of \eqref{tWK} is \Mh in
$z$ and its boundary value as $z\to\zeta\in\pBn$ coincides with $P(w,\zeta)$. To~see the former,
denote temporarily
\begin{align*}
a_{pqjm} &:= \frac{(n)_{j+p}(n)_{j+q}(n)_{m+p}(n)_{m+q}} {(n)_{j+p+q+m}}, \\
I_{pqj} &:= (1-|z|^2)^n \frac{\spr{z,w}^p \spr{w,z}^q |z|^{2j}}{p!q!j!}, \\
W_m &:= (1-|w|^2)^n \frac{|w|^{2m}}{m!},
\end{align*}
so~that
$$ K\Sz(z,w) = \frac{\Gamma(n)}{2\pi^n} \sum_{p,q,j,m=0}^\infty a_{pqjm}I_{pqj}W_m.  $$
By~a routine computation (here $\td$ applies to the $z$ variable)
$$ \frac{\td I_{pqj}}{1-|z|^2} = |w|^2 I_{p-1,q-1,j} - (j+n+p)(j+n+q) I_{pqj}
 +(p+q+n+j-1) I_{p,q,j-1} .  $$
Consequently, with the understanding that $I_{pqj}\equiv0$ if any of the subscripts
$p,q,j$ is negative,
\begin{align}
& \frac{2\pi^n/\Gamma(n)}{1-|z|^2} \td_z K\Sz(z,w) = |w|^2 \sum_{pqjm} a_{pqjm} I_{p-1,q-1,j}W_m
  \nonumber \\
&\hskip4em - \sum_{pqjm} (j+n+p)(j+n+q) a_{pqjm} I_{pqj} W_m  \nonumber \\
&\hskip4em + \sum_{pqjm} (p+q+n+j-1) a_{pqjm} I_{p,q,j-1} W_m  \nonumber  \\
&= |w|^2 \sum_{pqjm} a_{pqjm} I_{p-1,q-1,j}W_m
  \nonumber \\
&\hskip4em - \sum_{pqjm} (j+n+p)(j+n+q) a_{pqjm} I_{pqj} W_m  \nonumber \\
&\hskip4em + \sum_{pqjm} (p+q+n+j) a_{p,q,j+1,m} I_{pqj} W_m  \label{tWQ}.
\end{align}
Since $a_{p,q,j+1.m}=\frac{(n+j+p)(n+j+q)}{n+p+q+m+j}a_{pqjm}$, the second and third sums
combine into
\begin{align*}
& - \sum_{pqjm} \frac{m(n+p+j)(n+q+j)}{n+p+q+j+m} a_{pqjm} I_{pqj} W_m \\
&\hskip2em = - |w|^2 \sum_{pqjm} \frac{(n+p+j)(n+q+j)}{n+p+q+j+m} a_{pqjm} I_{pqj} W_{m-1} \\
&\hskip2em = -|w|^2 \sum_{pqjm} \frac{(n+p+j)(n+q+j)}{n+p+q+j+m+1} a_{pqj,m+1} I_{pqj} W_m .
\end{align*}
Since
\begin{align*}
&\frac{(n+p+j)(n+q+j)}{n+p+q+j+m+1} a_{pqj,m+1} \\
&\hskip2em = \frac{(n+p+j)(n+q+j)(n+m+p)(n+m+q)}{(n+p+q+j+m)(n+p+q+j+m+1)} a_{pqjm} \\
&\hskip2em = a_{p+1,q+1,j,m},
\end{align*}
this exactly cancels the first sum
in~\eqref{tWQ}. Thus, indeed, $\td_z K\Sz(z,w)\equiv0$.

To~verify the latter claim, let us carry out the summation over $j$ in~\eqref{tWO}:
\[ \begin{aligned}
& \frac{2\pi^n}{\Gamma(n)} K\Sz(z,w) = (1-|z|^2)^n (1-|w|^2)^n \sum_{pqm}
 \frac{(n)_p(n)_q(n)_{m+p}(n)_{m+q}}{(n)_{p+q+m}} \\
&\qquad\times \frac{\spr{z,w}^p\spr{w,z}^q|w|^{2m}}{p!q!m!}
 \FF21{n+p,n+q}{n+p+q+m}{|z|^2} . \end{aligned} \label{tWR}  \]
Using the standard Euler transformation formula for ${}_2\!F_1$ \cite[\S2.1(23)]{BE}
\[ \FF21{a,b}ct = (1-t)^{c-a-b} \FF21{c-a,c-b}ct,  \label{tWV}  \]
the right-hand side of \eqref{tWR} becomes
\begin{align*}
& (1-|w|^2)^n \sum_{pqm}
 \frac{(n)_p(n)_q(n)_{m+p}(n)_{m+q}}{(n)_{p+q+m}} \frac{\spr{z,w}^p\spr{w,z}^q|w|^{2m}}{p!q!m!} \\
&\hskip4em\times (1-|z|^2)^m \FF21{p+m,q+m}{n+p+q+m}{|z|^2} .
\end{align*}
As $z\to\zeta\in\pBn$, only the terms with $m=0$ survive, yielding
\begin{align*}
& (1-|w|^2)^n \sum_{pq} \frac{(n)_p^2(n)_q^2} {(n)_{p+q}} \frac{\spr{\zeta,w}^p\spr{w,\zeta}^q}{p!q!}
 \FF21{p,q}{n+p+q}1  \\
&= (1-|w|^2)^n \sum_{pq} \frac{(n)_p^2(n)_q^2} {(n)_{p+q}} \frac{\spr{\zeta,w}^p\spr{w,\zeta}^q}{p!q!}
 \GG{n+p+q,n}{n+p,n+q} \\
&= (1-|w|^2)^n \sum_{pq} \frac{(n)_p(n)_q} {p!q!} \spr{\zeta,w}^p\spr{w,\zeta}^q  \\
&= (1-|w|^2)^n (1-\spr{\zeta,w})^{-n} (1-\spr{w,\zeta})^{-n} = \frac{2\pi^n}{\Gamma(n)} P(w,\zeta),
\end{align*}
completing the proof.   \qed
\end{remark}

The~function $FD_1$ is known to satisfy an Euler-type transformation formula \cite[formula~7.1.(4)]{Karls}
\[ \begin{aligned}
& \FD{a,a',b_1,b_2}c{x_1,x_2,y_1,y_2} = (1-x_1)^{-b_1} (1-x_2)^{-b_2} \\
&\qquad\times
 \FD{c-a-a',a',b_1,b_2}c{\frac{x_1}{x_1-1},\frac{x_2}{x_2-1},\frac{y_1-x_1}{1-x_1},\frac{y_2-x_2}{1-x_2}}.
 \end{aligned} \label{tWL}  \]
On~the other hand, from its definition \eqref{tWJ} it is apparent that $FD_1$ enjoys the symmetry property
\[ \FD{a,a',b_1,b_2}c{x_1,x_2,y_1,y_2} = \FD{b_1,b_2,a,a'}c{x_1,y_1,x_2,y_2} . \label{tWM}  \]
These formulas can be used to simplify~\eqref{tWK}.

\begin{theorem} \label{PD}
For any $n\ge1$, $K\Sz(z,w)$ equals
\[ \begin{aligned}
& \frac{\Gamma(n)}{2\pi^n} \frac{(1-|w|^2)^n}{|1-\spr{z,w}|^{2n}} \sum_{i_1=0}^n \sum_{i_2,j_1=0}^{n-i_1}
 \frac{(-n)_{i_1+i_2} (-n)_{i_1+j_1} (n)_{i_2} (n)_{j_1}} {i_1!i_2!j_1!(n)_{i_1+i_2+j_1}} \\
&\hskip4em\times t_1^{i_1} t_2^{i_2} t_3^{j_1} \FF21{i_2+n,j_1+n}{i_1+i_2+j_1+n}{t_4}, \end{aligned}  \label{tWN}  \]
where
\[ t_1=|z|^2, \quad t_2=\frac{|z|^2-\spr{w,z}}{1-\spr{w,z}}, \quad t_3=\overline{t_2},
\quad t_4=1-\frac{(1-|z|^2)(1-|w|^2)}{|1-\spr{z,w}|^2} .  \label{tWO}  \]
\end{theorem}

\begin{proof} Applying \eqref{tWM} to the right-hand side of \eqref{tWL} gives
\begin{align*}
& \FD{a,a',b_1,b_2}c{x_1,x_2,y_1,y_2} = (1-x_1)^{-b_1} (1-x_2)^{-b_2} \\
&\qquad\times \FD{b_1,b_2,c-a-a',a'}c{\frac{x_1}{x_1-1},\frac{y_1-x_1}{1-x_1},\frac{x_2}{x_2-1},\frac{y_2-x_2}{1-x_2}}.
\end{align*}
Now we apply \eqref{tWM} one more time, to~the last right-hand side;
after a small computation, this yields
\begin{align*}
& \FD{a,a',b_1,b_2}c{x_1,x_2,y_1,y_2} = (1-x_1)^{c-a-b} (1-x_2)^{-b_2} (1-y_1)^{-a'} \\
&\;\times \FD{c-b_1-b_2,b_2,c-a-a',a'}c{x_1,\frac{x_1-y_1}{1-y_1},\frac{x_1-x_2}{1-x_2},
 \frac{x_1+y_2-x_2-y_1+x_2y_1-x_1y_2}{(1-x_2)(1-y_1)}} .
\end{align*}
Multiplying both sides by $(1-x_1)^n(1-y_2)^n$ and setting $a=a'=b_1=b_2=c=n$,
$x_1=|z|^2$, $x_2=\spr{z,w}$, $y_1=\spr{w,z}$, $y_2=|w|^2$, we~thus get from~\eqref{tWK}
\[ K\Sz(z,w) = \frac{\Gamma(n)}{2\pi^n} \frac{(1-|w|^2)^n}{|1-\spr{z,w}|^{2n}}
 \FD{-n,n,-n,n}n{t_1,t_2,t_3,t_4} \label{tWP}  \]
with $t_1,t_2,t_3,t_4$ as in~\eqref{tWO}. Finally, by~\eqref{tTJ}
\begin{align*}
& \FD{-n,n,-n,n}n{t_1,t_2,t_3,t_4} = \sum_{\substack{i_1+i_2\le n,\\i_1+j_1\le n,\\j_2\ge0}}
 \frac{t_1^{i_1} t_2^{i_2} t_3^{j_1} t_4^{j_2}}{i_1!i_2!j_1!j_2!}
 \frac{(-n)_{i_1+i_2} (n)_{j_1+j_2} (-n)_{i_1+j_1} (n)_{i_2+j_2}}{(n)_{i_1+i_2+j_1+j_2}} \\
&\hskip2em= \sum_{\substack{i_1+i_2\le n,\\i_1+j_1\le n}}
 \frac{t_1^{i_1} t_2^{i_2} t_3^{j_1}}{i_1!i_2!j_1!}
 \frac{(-n)_{i_1+i_2} (n)_{j_1} (-n)_{i_1+j_1} (n)_{i_2}}{(n)_{i_1+i_2+j_1}}
 \sum_{j_2\ge0} \frac{(n+j_1)_{j_2}(n+i_2)_{j_2}} {(n+i_1+i_2+j_1)_{j_2}} \frac{t_4^{j_2}}{j_2!}  \\
&\hskip2em= \sum_{\substack{i_1+i_2\le n,\\i_1+j_1\le n}}
 \frac{t_1^{i_1} t_2^{i_2} t_3^{j_1}}{i_1!i_2!j_1!}
 \frac{(-n)_{i_1+i_2} (n)_{j_1} (-n)_{i_1+j_1} (n)_{i_2}}{(n)_{i_1+i_2+j_1}}
 \FF21{n+j_1,n+i_2}{n+i_1+i_2+j_1}{t_4} .
\end{align*}
Substituting this into \eqref{tWP} yields \eqref{tWN}, completing the proof.  %\qed
\end{proof}

Using the formula \cite[\S2.10(11)]{BE}
\[ \begin{aligned}
& \FF21{n+1,n+m+1}{n+m+l+2}z = \frac{(n+m+l+1)!(-1)^{m+1}}{l!n!(m+n)!(m+l)!} \\
&\qquad\times \frac{d^{n+m}}{dz^{n+m}} \Big[ (1-z)^{m+l} \frac{d^l}{dz^l} \frac{\log(1-z)}z \Big],
 \qquad m,n,l=0,1,2,\dots, \end{aligned}  \label{tWU}  \]
it~is possible to express each ${}_2\!F_1$ in \eqref{tWN} in terms of $\log(1-t_4)$ and rational functions of $t_4$.

For~instance, for $n=1$ we get in this way
$$ \FD{1,1,1,1}1{t_1,t_2,t_3,t_4} = \frac{t_2t_3+(1-t_2-t_3)t_4}{t_4(1-t_4)}
 + \frac{t_2t_3-t_1t_4}{t_4^2} \log(1-t_4).  $$
For $t_1=|z|^2$, $t_2=\oz\frac{z-w}{1-\oz w}$, $t_3=\overline t_2$ and
$t_4=1-\frac{(1-|z|^2)(1-|w|^2)}{|1-w\oz|^2}$, the right-hand side simplifies
just~to $\frac{1-|z|^2|w|^2}{1-|w|^2}$, implying that $K\Sz(z,w)=\frac1{2\pi}
\frac{1-|z|^2|w|^2}{|1-w\oz|^2}$, in~complete accordance with~\eqref{tTI}.
For $n=2$, the formula for $K\Sz$ already becomes quite complicated, without any
apparent possibility of simplification.

Note also that the \Mh Szeg\"o kernel $K\Sz(z,w)$ is symmetric in $z,w$,
though this is not visible at all from the formula~\eqref{tWN}.

We~conclude this section by discussing the smoothness of $K\Sz$ on the closure
of $\Bn\times\Bn$.

\begin{proposition} \label{PE}
For $n>1$,
$$ K\Sz\in C^{n-1}(\bbd) \setminus C^n(\bbd).  $$
\end{proposition}

\begin{proof} Let $U$ be a neighborhood of $\operatorname{diag}(\pBn)$.
Then $1-\spr{z,w}$ stays away from 0 on $\overline{\Bn\times\Bn}\setminus U$.
Keeping our previous notation $x_1=|z|^2$, $x_2=\spr{z,w}$, $y_1=\spr{w,z}$, $y_2=|w|^2$,
and denoting in addition temporarily $Q:=1/(1-\spr{z,w})$, we~have $t_2=\overline t_3=
1-(1-x_1)Q$, so~from~\eqref{tWN}
\[ \begin{aligned}
K\Sz(z,w) &= (1-y_2)^n |Q|^{2n} \frac{\Gamma(n)}{2\pi^n} \sum_{i_1=0}^n
 \sum_{i_2,j_1=0}^{n-i_1} x_1^{i_1} (1-(1-x_1)Q)^{i_2} (1-(1-x_1)\overline Q)^{j_1} \\
&\qquad\times \frac{(-n)_{i_1+i_2} (n)_{j_1} (-n)_{i_1+j_1} (n)_{i_2}}{(n)_{i_1+i_2+j_1}}
 \FF21{n+j_1,n+i_2}{n+i_1+i_2+j_1}{1-(1-x_1)(1-y_2)|Q|^2} . \end{aligned}  \label{tWS}  \]
By~the standard formulas for the analytic continuation of the hypergeometric functions
${}_2\!F_1$ \cite[\S2.10(14)]{BE}, we~have for any $a,b>0$ and $m=0,1,2,\dots$,
$$ \FF21{a,b}{a+b-m}z = (1-z)^{-m} A_{abm}(1-z) + B_{abm}(1-z)\log(1-z),  $$
with some functions $A_{abm},B_{abm}$ holomorphic on~$\DD$. The right-hand side of \eqref{tWS}
therefore equals, omitting for a moment the factor $\frac{\Gamma(n)}{2\pi^n}$,
\begin{align*}
& \sum_{i_1=0}^n \sum_{i_2,j_1=0}^{n-i_1}
 x_1^{i_1} (1-(1-x_1)Q)^{i_2} (1-(1-x_1)\overline Q)^{j_1}
 \frac{(-n)_{i_1+i_2} (n)_{j_1} (-n)_{i_1+j_1} (n)_{i_2}}{(n)_{i_1+i_2+j_1}}  \\
&\qquad\times \Big[ \frac{(1-y_2)^{i_1}|Q|^{2i_1}}{(1-x_1)^{n-i_1}} A +
  (1-y_2)^n|Q|^{2n} B \log[(1-x_1)(1-y_2)|Q|^2] \Big] ,
\end{align*}
where $A\equiv A_{n+j_1,n+i_2,n-i_1}(1-1(1-x_1)(1-y_2)|Q|^2)$
and $B\equiv B_{n+j_1,n+i_2,n-i_1}(1-1(1-x_1)(1-y_2)|Q|^2)$.
In~terms of Taylor series around $(x_1,y_2)=(1,1)$, the~last expression has the form
\[ \begin{aligned}
& \sum_{j=-n}^\infty \sum_{k,l,m=0}^\infty a_{jklm}(1-x_1)^j(1-y_2)^k Q^l\overline Q^m \\
&+ \sum_{k=n}^\infty \sum_{j,l,m=0}^\infty b_{jklm}(1-x_1)^j(1-y_2)^k Q^l\overline Q^m
  \log[(1-x_1)(1-y_2)|Q|^2] , \end{aligned}  \label{tWT}  \]
with some coefficients $a_{jklm},b_{jklm}$. However, since $K\Sz(z,w)$ is symmetric in~$z,w$,
\eqref{tWT}~remains unchanged upon replacing $x_1,y_2$ and $Q$ by $y_2,x_1$ and $\overline Q$,
respectively. Consequently, $a_{jkml}$ must vanish for $j<0$, and $b_{jklm}$ must vanish for $j<n$.
It~follows that the first sum in \eqref{tWT} is $C^\infty$ on $(1-x_1,1-y_2,Q)\in\DD\times\DD\times(\CC\setminus\{0\})$,
while the second sum is $C^{n-1}$ there. (Note that $Q$ is bounded away from zero,
in~fact $|Q|>\frac12$.) This in turn means that the right-hand side of \eqref{tWS}
is $C^{n-1}$ on $\overline{\Bn\times\Bn}\setminus U$. Thus, indeed, $K\Sz\in C^{n-1}(\bbd)$.

It~remains to show that $K\Sz$ does not belong to~$C^n(\bbd)$. If~it~did, then the function
$$ \int_\pBn K\Sz(z,R\xi) H^{pq}(\cdd\xi\eta) \,d\sigma(\xi)  $$
would belong to $C^n(\overline{\Bn})$, for any fixed $0<R<1$, $\eta\in\pBn$ and $p,q\ge0$.
However, by \eqref{tVJ} and the orthogonality relations~\eqref{tVD}, the last integral equals
$S^{pq}(|z|)S^{pq}(R)H^{pq}(\cdd{\frac z{|z|}}\eta)$, and the hypergeometric function
$S^{pq}(t)$ is well known not to be $C^n$ at the point $t=1$ (it~again contains a logarithmic
singularity of the form $(1-t)^n\log(1-t)$). This completes the proof.  %\qed
\end{proof}

\section{General \Mh kernels}  \label{sec4}
We~now turn to general $\Un$-invariant measures $d\mu\otimes d\sigma$ on $\Bn$ and their associated
\Mh kernels~$K_\mu$.

\begin{theorem} \label{PF}
For any $n\ge1$ and $\mu$ as in~\eqref{tTK}, $K_\mu$ is given~by
\[ K_\mu(z,w) = \frac{\Gamma(n)}{2\pi^n} (1-|z|^2)^n (1-|w|^2)^n
 \sum_{p,q,j,m=0}^\infty A_{pqjm}(\mu)
 \frac{\spr{z,w}^p \spr{w,z}^q |z|^{2j} |w|^{2m}}{p!q!j!m!} ,  \label{tXA}  \]
where
\[ \begin{aligned}
& A_{pqjm}(\mu) := \sum_{l=0}^{\min(m,j)}
 \frac{\Gamma(n+p+j)\Gamma(n+q+j)} {\Gamma(n)\Gamma(n+p+q+j+l)}
 \frac{\Gamma(n+p+m)\Gamma(n+q+m)} {\Gamma(n)\Gamma(n+p+q+m+l)}  \\
&\qquad\times \frac{(-1)^l \Gamma(n+p+q+l-1)(n+p+q+2l-1)(-j)_l(-m)_l}
   {\Gamma(n)l! c_{p+l,q+l}(\mu)}, \end{aligned}  \label{tXB}  \]
with
\[ c_{pq}(\mu) := \frac{\Gamma(p+n)^2\Gamma(q+n)^2} {\Gamma(n)^2\Gamma(p+q+n)^2}
 \int_0^1 t^{p+q+n-1} \FF21{p,q}{p+q+n}t ^2 \,d\mu(t).   \label{tXC}  \]
\end{theorem}

\begin{proof} Clearly, the space $\HH=L^2\mh(\Bn,d\mu\otimes d\sigma)$ satisfies
the hypotheses of Proposition~\ref{PA}. For~any pair of functions $u(r\zeta)=S^{pq}(r)f(\zeta)$
and $v(r\zeta)=S^{pq}(r)g(\zeta)$ in~$\bhpq$, we~have
\begin{align*}
\spr{u,v}_\HH &= \int_0^1 \int_\pBn S^{pq}(\sqrt t) f(\zeta) S^{pq}(\sqrt t) \overline{g(\zeta)} \,(d\mu\otimes d\sigma)(t,\zeta) \\
&= \spr{f,g}_{L^2(\pBn,d\sigma)} \int_0^1 S^{pq}(\sqrt t)^2 t^{n-1} d\mu(t),
\end{align*}
so \eqref{tVN} holds with
$$ c_{pq} = \GG{p+n,q+n}{n,p+q+n}^2 \int_0^1 t^{p+q+n-1} \FF21{p,q}{p+q+n}t  ^2 \,d\mu(t) \equiv c_{pq}(\mu),  $$
by~\eqref{tXC}. Consequently, by~\eqref{tVO}, for $z=r\xi$ and $w=R\eta$,
\[ K_\mu(z,w) = \sum_{p,q} \frac{S^{pq}(r)S^{pq}(R) H^{pq}(\cdd\xi\eta)}{c_{pq}(\mu)}. \label{tXD}  \]
To~simplify the notation, let us temporarily denote $\cdd\xi\eta=:\zeta\in\DD$ and
$$ s^{pq}(t) = \GG{n+p,n+q}{n,n+p+q} \FF21{p,q}{n+p+q}t,  $$
so~that $S^{pq}(r)=r^{p+q} s^{pq}(r^2)$. Let~us first consider the sum in \eqref{tXD} over terms with $p\ge q$
--- say, $p=q+r$, $r\ge0$. Using~\eqref{tVE}, the~sum becomes (omitting momentarily the constant factor $\frac{\Gamma(n)}{2\pi^n}$)
\begin{align}
&\sum_{q,r=0}^\infty \frac{|z|^{2q+r} s^{q+r,q}(|z|^2)|w|^{2q+r} s^{q+r,q}(|w|^2)}{c_{q+r,q}(\mu)} \;\times \nonumber \\
&\hskip2em  \frac{(-1)^q(n+2q+r-1)(n+q+r-2)!}{\Gamma(n)q!r!} \ozeta^r
  \sum_{j=0}^q \frac{(-q)_j(n+q+r-1)_j}{(r+1)_j j!} \zeta^j \ozeta^j  \nonumber \\
&= \sum_{q,r=0}^\infty \sum_{j=0}^q
 \frac{|z|^{2q+r} s^{q+r,q}(|z|^2)|w|^{2q+r} s^{q+r,q}(|w|^2)}{c_{q+r,q}(\mu)} \;\times \nonumber \\
&\hskip4em \frac{(-1)^q (n+2q+r-1)(n+q+r+j-2)!(-q)_j}{\Gamma(n)q!(r+j)!j!} \zeta^j \ozeta^{j+r} . \label{tXF}
\end{align}
Letting $q=j+l$ and noticing that $(-q)_j/q!=(-1)^j/(q-j)!$, this becomes
\begin{align*}
& \sum_{j,l,r\ge0} \frac{|z|^{2j+2l+r}s^{j+l+r,j+l}(|z|^2) |w|^{2j+2l+r}s^{j+l+r,j+l}(|w|^2)} {c_{j+l+r,j+l}(\mu)}  \\
&\hskip4em \times \frac{(-1)^l(n+2j+2l+r-1)(n+2j+l+r-2)!}{\Gamma(n)(r+j)!j!l!} \zeta^j \ozeta^{j+r} ,
\end{align*}
or, writing $q$ and $p=q+r$ instead of $j$ and $j+r$, respectively,
\begin{align*} \sum_{\substack{q,r\ge0\\p=q+r}} \sum_{l\ge0}
& \frac{|z|^{p+q+2l} s^{p+l,q+l}(|z|^2) |w|^{p+q+2l} s^{p+l,q+l}(|w|^2)}{c_{p+l,q+l(\mu)}} \\
&\hskip4em \times \frac{(-1)^l(n+p+q+2l-1)(n+p+q+l-2)!}{\Gamma(n)p!q!l!} \zeta^q\ozeta^p .
\end{align*}
The sum over $p<q$ in \eqref{tXD} is treated in the same way, and recalling that $\spr{z,w}=|z||w|\zeta$,
we~thus arrive~at
\begin{align*}
K_\mu(z,w) &= \frac{\Gamma(n)}{2\pi^n} \sum_{p,q,l=0}^\infty
 \frac{|z|^{2l}s^{p+l,q+l}(|z|^2)|w|^{2l}s^{p+l,q+l}(|w|^2)}{c_{p+l,q+l}(\mu)} \;\times \\
&\hskip2em \frac{(-1)^l(n+p+q+2l-1)(n+p+q+l-2)!}{\Gamma(n)l!} \frac{\spr{z,w}^q\spr{w,z}^p}{q!p!}.
\end{align*}
On~the other hand, by~the Euler transformation formula~\eqref{tWV}, we~have
\begin{align*}
s^{pq}(t) &= \GG{p+n,q+n}{n,p+q+n}(1-t)^n \FF21{p+n,q+n}{p+q+n}t \\
&= (1-t)^n \sum_{k=0}^\infty \GG{p+n+k,q+n+k}{p+q+n+k,n} \frac{t^k}{k!}.
\end{align*}
Consequently,
$$ K_\mu(z,w) = (1-|z|^2)^n(1-|w|^2)^n \frac{\Gamma(n)}{2\pi^n}
  \sum_{p,q,j,m=0}^\infty A_{pqjm}(\mu) \frac{|z|^{2j}|w|^{2k}}{j!k!}
  \frac{\spr{z,w}^q\spr{w,z}^p}{q!p!}  $$
with
\begin{align*}
A_{pqjm}(\mu) &= \sum_{\substack{l,i,k\ge0,\\l+i=j,\\l+k=m}}
 \frac{(-1)^l(n+p+q+2l-1)(n+p+q+l-2)!}{\Gamma(n)l! c_{p+l,q+l}(\mu)} \;\times \\
&\hskip2em \GG{p+l+n+i,q+l+n+i}{n,p+q+n+i+2l} \GG{p+l+n+k,q+l+n+k}{n,p+q+n+k+2l}
 \frac{j!}{i!} \frac{m!}{k!}  \\
&= \sum_{l=0}^{\min(m,j)}
 \frac{(-1)^l(n+p+q+2l-1)(n+p+q+l-2)!}{\Gamma(n)l! c_{p+l,q+l}(\mu)} \;\times \\
&\hskip2em \GG{p+n+j,q+n+j}{n,p+q+n+j+l} \GG{p+n+m,q+n+m}{n,p+q+n+m+l} (-j)_l (-m)_l,
\end{align*}
since $j!/(j-l)!=(-1)^l(-j)_l$ and similarly for $m!/(m-l)!$.
But~this is precisely \eqref{tXA} and~\eqref{tXB}, completing the proof.   %\qed
\end{proof}

\begin{remark} Taking for $d\mu(t)$ the point mass at $t=1$, one~can use the last
theorem to give an independent proof of the formula \eqref{tWK} for the \Mh Szeg\"o
kernel (not~using Theorem~\ref{PA}). Indeed, in~that case $c_{pq}=1$ for all~$p,q$,
so, pulling out some Gamma functions,
\begin{align*}
& A_{pqjm}(\mu) = \GG{n+p+j,n+q+j}{n,n+p+q+j} \GG{n+p+m,n+q+m}{n,n+p+q+m}  \\
&\hskip2em \sum_{l=0}^{\min(m,j)} \frac{(-1)^l(n+p+q+2l-1)(n+p+q+l-2)!}{\Gamma(n)l!}
 \frac{(-j)_l(-m)_l}{(p+q+n+j)_l(p+q+n+m)_l} .
\end{align*}
Denoting momentarily for brevity $n+p+q=:h$, we~have
$$ n+p+q+2l-1 = 2\Big(\frac{h-1}2+l\Big) = (h-1) \frac{(\frac{h+1}2)_l}{(\frac{h-1}2)_l} , $$
so~the last sum can be written~as
\begin{align}
& \frac{h-1}{\Gamma(n)} \sum_l \frac{(-j)_l(-m)_l}{(h+j)_l(h+m)_l} \frac{(-1)^l}{l!}
 \frac{(\frac{h+1}2)_l}{(\frac{h-1}2)_l} \Gamma(h+l-1)  \nonumber  \\
&= \frac{\Gamma(h)}{\Gamma(n)} \FF43{-j,-m,\frac{h+1}2,h-1}{h+j,h+m,\frac{h-1}2}{-1}.
 \label{tXE}
\end{align}
Now by a formula of Bailey \cite[\S4.5(4)]{BE}
$$ \FF43{a,1+\frac a2,b,c}{\frac a2,1+a-b,1+a-c}{-1} = \GG{1+a-b,1+a-c}{1+a,1+a-b-c},  $$
so \eqref{tXE} is equal~to
$$ \frac{\Gamma(h)}{\Gamma(n)} \GG{h+j,h+m}{h,h+j+m} = \GG{h+j,h+m}{n,h+j+m}  $$
and
\begin{align*}
A_{pqjm}(\mu) &= \GG{n+p+j,n+q+j}n \GG{n+p+m,n+q+m}n \GG{-}{n,n+p+q+j+m} \\
&= \frac{(n)_{p+j}(n)_{q+j}(n)_{p+m}(n)_{q+m}}{(n)_{p+q+j+m}},
\end{align*}
proving the claim.  \qed
\end{remark}

\begin{remark} \label{reFo}
A~similar argument as in the proof of Theorem~\ref{PF} can also be used to give another proof
of Folland's formula \eqref{tVI} for the \Mh Poisson kernel: for $z=r\xi$,
\[ \sum_{p,q} S^{pq}(|z|) H^{pq}(\cdd\xi\eta) = \frac{\Gamma(n)}{2\pi^n} \frac{(1-|z|^2)^n}{|1-\spr{z,\eta}|^{2n}}. \label{tXG}  \]
Indeed, assuming again first that $p\ge q$ --- say, $p=q+r$, $r\ge0$ --- and using~\eqref{tVE},
the~sum over $p\ge q$ becomes, as~in~\eqref{tXF} (omitting yet again temporarily the constant
factor $\Gamma(n)/2\pi^n$, and denoting again $\cdd\xi\eta=:\zeta$)
$$ \sum_{q,r=0}^\infty \sum_{j=0}^q |z|^{2q+r} s^{q+r,q}(|z|^2)
 \frac{(-1)^q (n+2q+r-1)(n+q+r+j-2)!(-q)_j}{\Gamma(n)q!(r+j)!j!} \zeta^j \ozeta^{j+r} , $$
or, upon setting $q=j+l$,
$$ \sum_{j,l,r\ge0} |z|^{2j+2l+r}s^{j+l+r,j+l}(|z|^2)
 \frac{(-1)^l(n+2j+2l+r-1)(n+2j+l+r-2)!}{\Gamma(n)(r+j)!j!l!} \zeta^j \ozeta^{j+r} ,  $$
that~is, writing $q$ and $q=p+r$ instead of $j$ and $j+r$, respectively,
$$ \sum_{\substack{q,r\ge0\\p=q+r}} \sum_{l\ge0}
 |z|^{p+q+2l} s^{p+l,q+l}(|z|^2)
 \frac{(-1)^l(n+p+q+2l-1)(n+p+q+l-2)!}{\Gamma(n)p!q!l!} \zeta^q\ozeta^p .  $$
The sum over $p<q$ is treated in the same way, and we see that the right-hand side of \eqref{tXG} equals
\begin{align*}
& \frac{\Gamma(n)}{2\pi^n} \sum_{p,q,l=0}^\infty |z|^{2l}s^{p+l,q+l}(|z|^2)
 \frac{(-1)^l(n+p+q+2l-1)(n+p+q+l-2)!}{\Gamma(n)l!} \frac{\spr{z,\eta}^q\spr{\eta,z}^p}{q!p!} \\
&= \frac{\Gamma(n)}{2\pi^n} (1-|z|^2)^n \sum_{p,q,m=0}^\infty A_{pqm} \frac{|z|^{2m}}{m!}
 \frac{\spr{z,\eta}^q\spr{\eta,z}^p}{q!p!} ,
\end{align*}
where
\begin{align*}
A_{pqm} &= \sum_{\substack{l,k\ge0\\l+k=m}} \frac{(-1)^l(n+p+q+2l-1)(n+p+q+l-2)!}{\Gamma(n)l!}
 \GG{p+l+n+k,q+l+n+k}{n,p+q+n+k+2l} \frac{m!}{k!}  \\
&= \sum_{l=0}^m \frac{(n+p+q+2l-1)(n+p+q+l-2)!}{\Gamma(n)l!}
 \GG{p+n+m,q+n+m}{n,p+q+n+m+l} (-m)_l  \\
&= \GG{p+n+m,q+n+m}{n,p+q+n+m} \sum_{l=0}^m \frac{(n+p+q+2l-1)(n+p+q+l-2)!}{\Gamma(n)l!}
 \frac{(-m)_l}{(p+q+n+m)_l}  \\
&= \GG{p+n+m,q+n+m}{n,h+m} \GG hn \FF32{h-1,\frac{h+1}2,-m}{\frac{h-1}2,h+m}1,
\end{align*}
as~in~\eqref{tXE}; here we have again set $n+p+q=:h$. Now~by a formula due to Dixon \cite[\S4.4(5)]{BE}
$$ \FF32{a,b,c}{1+a-b,1+a-c}1 = \GG{1+\frac a2,1+a-b,1+a-c,1+\frac a2-b-c}{1+a.1+\frac a2-b,1+\frac a2-c,1+a-b-c}, $$
so the penultimate ${}_3\!F_2$ is equal~to
$$ \GG{\frac{h+1}2,\frac{h-1}2,h+m,m}{h,0,\frac{h+1}2+m,\frac{h-1}2+m}.  $$
This vanishes for $m\ge1$, and reduces to $1$ for $m=0$. Hence $A_{pqm}=0$ for $m\ge1$, while
$$ A_{pq0} = \GG{p+n,q+n}{n,p+q+n} \GG{p+q+n}n = (n)_p (n)_q ,  $$
so~that
$$ \sum_{p,q,m=0}^\infty A_{pqm} \frac{|z|^{2m}}{m!} \frac{\spr{z,\eta}^q\spr{\eta,z}^p}{q!p!}
= \sum_{p,q} (n)_p (n)_q \frac{\spr{z,\eta}^q\spr{\eta,z}^p}{q!p!} = (1-\spr{z,\eta})^{-n}(1-\spr{\eta,z})^{-n},  $$
proving~\eqref{tXG}.  \qed
\end{remark}

Choosing $d\mu(t)=\frac12(1-t)^s\,dt$, the last theorem gives, in~particular, a~formula for the
weighted \Mh Bergman kernels~$K_s$, $s>-1$. The~coefficients $c_{pq}(\mu)$ are then given~by
\[ c_{pq}(s) := \frac12 \GG{n+p,n+q}{n,n+p+q}^2 \int_0^1 t^{p+q+n-1} \FF21{p,q}{p+q+n}t ^2 (1-t)^s \,dt.  \label{tXH}  \]
We~conclude this section by evaluating $c_{00}$ and $c_{01}$ in the particular case of
the unweighted \Mh Bergman kernel $s=0$ in dimension $n=2$.

For~$c_{00}$, we~actually have quite generally
\[ c_{00}(s) = \frac12 \int_0^1 t^{n-1}(1-t)^s \,dt = \frac{\Gamma(n)\Gamma(s+1)}{2\Gamma(n+s+1)}.  \label{tXK}  \]
For $p=q=1$ and $s=0$, \eqref{tXH} reads
$$ c_{11}(0) = \frac12 \frac{n^2}{(n+1)^2} \int_0^1 t^{n+1} \FF21{1,1}{n+2}t ^2 \,dt.  $$
Using \eqref{tWV} and~\eqref{tWU},
$$ \FF21{1,1}{n+2}t = (1-t)^n \FF21{n+1,n+1}{n+2}t = (1-t)^n \frac{(n+1)!}{n!^2} \frac{d^n}{dt^n} \frac{-\log(1-t)}t. $$
For $n=2$, a~short computation reveals that this reduces~to
\[ \FF21{1,1}4t = \frac{3t(3t-2)-6(1-t)^2\log(1-t)}{2t^3} . \label{tXJ}  \]
Differentiating the familiar Beta integral
$$ \int_0^1 t^a (1-t)^b \,dt = \frac{\Gamma(a+1)\Gamma(b+1)}{\Gamma(a+b+2)} ,
 \qquad a,b>-1,  $$
with respect to $b$ and setting $b=0$, we~get
\begin{align*}
& \int_0^1 t^a \log(1-t) \,dt = \frac{\psi(1)-\psi(a+2)}{a+1}, \\
& \int_0^1 t^a \log(1-t)^2 \,dt = \frac{(\psi(1)-\psi(a+2))^2+\psi'(1)-\psi'(a+2)}{a+1},
\end{align*}
where $\psi=\Gamma'/\Gamma$ is the logarithmic derivative of the Gamma function.
Squaring~\eqref{tXJ}, multiplying by $t^{x+3}$, $x>2$, and using the last two formulas yields
an (unwieldy) explicit formula for $\int_0^1 t^{x+3} \FF21{1,1}4t ^2\,dt$. A~tedious but utterly routine
calculation reveals that the result extends analytically to $\Re x>-4$ (as~it~should!),
and its value at $x=0$~is
\begin{align*}
& \frac{117}2\psi(1)-\frac{171}8+\frac92(\psi(3)-\psi(1))^2 -36(\psi(1)-\psi(2))^2
 +\frac{27}2\psi(3) - 72\psi(2) \\
&\hskip4em -\frac{63}2\psi'(1) +36\psi'(2) - \frac92\psi'(3)-54\psi''(1).
\end{align*}
Finally, recalling that for $k,m=1,2,3,\dots$,
\begin{align*}
\psi(m) &= -C + \sum_{j=1}^{m-1} \frac1j, \\
\psi^{(k)}(m) &= (-1)^k(k-1)!\zeta(k+1) +(-1)^k k!\sum_{j=1}^{m-1} \frac1{j^{k+1}} ,
\end{align*}
where $C$ is the Euler constant and $\zeta$ the Riemann zeta function,
we~finally get that for $n=2$
$$ c_{11}(0) = \frac{96\zeta(3)-115}4  $$
and
$$ \frac{c_{11}(0)}{c_{00}(0)} = 96\zeta(3) - 115.  $$
Since by~\eqref{tXB}
$$ A_{0000}(s) = \frac1{c_{00}(s)}, \qquad A_{1100}(s) = \frac{n^3}{(n+1)c_{11}(s)},  $$
we also get for $n=2$
$$ A_{0000}(0)=4, \qquad A_{1100}(0) = \frac{32}{96\zeta(3)-115}.  $$
It~is somewhat unlikely for a function whose Taylor coefficient involves $96\zeta(3)-115$
in the denominator to be given by some nice ``closed'' formula in terms of e.g.~hypergeometric
and similar functions. Thus there is probably no ``explicit'' expression for $K_0(z,w)$ when
$n=2$, hence \emph{a fortiori} also for $K_s(z,w)$ for general $s>-1$ and $n\ge2$.

\section{Asymptotics of $c_{pq}(s)$} \label{sec5}
Recall that in the polar coordinates $z=r\zeta$ on~$\CC^n$ ($r>0$, $\zeta\in\pBn$),
the Euclidean Laplacian $\Delta$ is given~by
$$ \Delta = \frac{\partial^2}{\partial r^2} + \frac{2n-1}r \frac\partial{\partial r}
 + \frac1{r^2} \dsph,  $$
where $\dsph$ is the \emph{spherical Laplacian}, which involves only differentiations
with respect to the $\zeta$ variables. In~particular, the value of $\dsph\phi$ on
a sphere $|z|=$const. depends only on the values of the function $\phi$ on that sphere.
Another operator with this property is the \emph{complex normal derivative}
(or~\emph{Reeb vector field})
$$ \cR := \sum_{j=1}^n \Big(z_j\frac\partial{\partial z_j}-\oz_j\frac\partial{\partial\oz_j}\Big). $$
Both $\dsph$ and $\cR$ commute with the action of~$\Un$, i.e.~$\dsph(\phi\circ U)=(\dsph\phi)\circ U$
for any $U\in\Un$, and similarly for~$\cR$. (In~fact, the algebra of all $\Un$-invariant linear
differential operators on $\pBn$ is generated by $\dsph$ and $\cR$, but we will not need this fact.)
From the irreducibility of the multiplicity-free decomposition \eqref{tVB} it follows by abstract
theory that $\dsph$ and $\cR$ map each $\chpq$ (and~$\bhpq$) into itself and actually reduce on it to
a multiple of the identity. Evaluation on e.g.~the element $\zeta_1^p\ozeta_2^q\in\chpq$ (for~$n\ge2$)
shows that, explicitly,
\[ \begin{aligned}
\dsph|\chpq &= -(p+q)(p+q+2n-2) I|\chpq, \\
\cR|\chpq &= (p-q)I|\chpq
\end{aligned}  \label{tYA}   \]
(which formulas prevail also for $n=1$; in~that case $\dsph=-\cR^2$). In~view of~\eqref{tVB},
both $\dsph$ and $\cR$ thus give rise to self-adjoint operators on $L^2(\pBn,d\sigma)$,
and the operator
$$ \cD := [-\dsph+(n-1)^2I]^{1/2}  $$
(in~the sense of functional calculus of self-adjoint operators) corresponds to multiplication
by $p+q$ on~$\chpq$. Consequently, the operators $\frac{\cD+\cR}2,\frac{\cD-\cR}2$ correspond
to multiplication on $\chpq$ by $p$ and~$q$, respectively, and for ``any'' double sequence
$\{f(p,q)\}_{p,q=0}^\infty$, the operator $f(\frac{\cD+\cR}2,\frac{\cD-\cR}2)$ (again taken in
the sense of functional calculus) will correspond to multiplication by $f(p,q)$ on~$\chpq$.
Taking, in~particular, $f(p,q)=1/c_{pq}(s)$ with the $c_{pq}(s)$ from~\eqref{tXH},
and denoting the corresponding operator $f(\frac{\cD+\cR}2,\frac{\cD-\cR}2)$ by~$M_s$,
we~thus deduce from \eqref{tVO} and~\eqref{tVJ} that
\[ K_s(z,w) = M_s K\Sz(z,w)  , \label{tYB} \]
where on the right-hand side we can choose to apply $M_s$ to either the $z$ or the $w$ variable
(the~result will be the same). Looking at the ``leading order term'' of~$M_s$, i.e.~the one corresponding
--- loosely speaking --- to~highest order derivatives, we~can thus get a rough idea of the behavior of
$K_s$ from that of~$K\Sz$ (which we are familiar with from Section~\ref{sec3}).
(This is the usual machinery of microlocal analysis.)

Since, in~view of~\eqref{tYA}, the~order of differentiation corresponds to the homogeneity
degree of $f(p,q)$ in~$(p,q)$, we~thus need to find the asymptotic behavior of $c_{pq}(s)$
as $p+q\to+\infty$.

\begin{theorem} \label{PG}
Let $p,q>0$ be fixed. Then as $\lambda\to+\infty$,
we~have the asymptotic expansion
\[ c_{\lambda p,\lambda q}(s) \approx \GG{2n+s+1,n+s+1,n+s+1,s+1}{n,n,2n+2s+2}
 \frac{\lambda^{-2s-2}}{(pq)^{s+1}} \sum_{j=0}^\infty \frac{a_j(p,q)}{\lambda^j}, \label{tYC}  \]
where $a_0(p,q)=1$.
\end{theorem}

\begin{proof} Inserting the standard representation for the ${}_2\!F_1$ function
$$ \FF21{a,b}cz = \GG c{b,c-b} \int_0^1 \frac{t^{b-1}(1-t)^{c-b-1}}{(1-tz)^a} \,dt,
 \qquad c>b>0,  $$
into \eqref{tXH} yields
$$ 2c_{pq}(s) = \GG{p+n,q+n}{n,n,p,q}\int_0^1\int_0^1\int_0^1 t^{p+q+n-1} (1-t)^s
 \frac{x^{q-1}(1-x)^{n+p-1}y^{p-1}(1-y)^{n+q-1}} {(1-tx)^p(1-ty)^q} \,dx\,dy\,dt.  $$
Using the representation \cite[\S5.8 (5)]{BE}
$$ \FE1{a;b,b'}c{x,y} = \GG c{a,c-a} \int_0^1 \frac{t^{a-1}(1-t)^{c-a-1}} {(1-tx)^b(1-ty)^{b'}} \,dt,
 \qquad c>a>0, $$
for the Appell $F_1$ function, this becomes
\begin{align*}
2c_{pq}(s) &= \GG{p+n,q+n}{n,n,p,q} \GG{p+q+n,s+1}{p+q+n+s+1} \int_0^1 \int_0^1
 x^{q-1}(1-x)^{n+p-1}y^{p-1}(1-y)^{n+q-1} \\
&\hskip4em\times \FE1{p+q+n;p,q}{p+q+n+s+1}{x,y}\,dx\,dy.
\end{align*}
By~the transformation formula for $F_1$ \cite[\S5.11(1)]{BE}
$$ \FE1{a;b,b'}c{x,y} = (1-x)^{-b}(1-y)^{-b'} \FE1{c-a;b,b'}c{\frac x{x-1},\frac y{y-1}}, $$
we~can continue with
\begin{align*}
2c_{pq}(s) &= \GG{p+n,q+n}{n,n,p,q} \GG{p+q+n,s+1}{p+q+n+s+1} \int_0^1 \int_0^1
 x^{q-1}(1-x)^{n-1}y^{p-1}(1-y)^{n-1} \\
&\hskip4em\times \FE1{s+1;p,q}{p+q+n+s+1}{\frac x{x-1},\frac y{y-1}}\,dx\,dy.
\end{align*}
Applying yet another representation formula for $F_1$ \cite[\S5.8(1)]{BE}
\begin{align*}
& \FE1{a;b,b'}c{x,y} = \GG c{b,b',c-b-b'} \iint_{\substack{u,v>0\\u+v<1}}
 \frac{u^{b-1} v^{b'-1} (1-u-v)^{c-b-b'-1}} {(1-ux-vy)^a} \,du\,dv,  \\
&\hskip12em \vphantom{\int} \qquad b,b'>0, \; c-b-b'>0,
\end{align*}
we~arrive~at
\begin{align*}
2c_{pq}(s) &= \GG{p+q+n,s+1,p+n,q+n}{n,n,p,p,q,q,n+s+1} \int_0^1 \int_0^1 \iint_{\substack{u,v>0\\u+v<1}}
 x^{q-1}(1-x)^{n-1}y^{p-1}(1-y)^{n-1} \\
& \hskip2em\times \frac{u^{p-1}v^{q-1}(1-u-v)^{n+s}}
  {(1+\frac{ux}{1-x}+\frac{vy}{1-y})^{s+1}} \,du\,dv\,dx\,dy .
\end{align*}
Performing the change of variable $u=\tau\sigma$, $v=(1-\tau)\sigma$, $du\,dv=\sigma\,d\sigma\,d\tau$,
this transforms into
\begin{align*}
2c_{pq}(s) &= \GG{p+q+n,s+1,p+n,q+n}{n,n,p,p,q,q,n+s+1} \int_0^1 \int_0^1 \int_0^1 \int_0^1
 x^{q-1}(1-x)^{n-1}y^{p-1}(1-y)^{n-1} \\
&\hskip4em\times \frac{\sigma^{p+q-1}(1-\sigma)^{n+s}\tau^{p-1}(1-\tau)^{q-1}}
  {(1+\frac{\sigma\tau x}{1-x}+\frac{\sigma(1-\tau)y}{1-y})^{s+1}} \,d\sigma\,d\tau\,dx\,dy  \\
&= \GG{p+q+n,s+1,p+n,q+n}{n,n,p,p,q,q,n+s+1} \int_0^1 \int_0^1 \int_0^1 \int_0^1
 x^{q-1}(1-x)^{n+s}y^{p-1}(1-y)^{n+s} \\
&\hskip2em\times\frac{\sigma^{p+q-1}(1-\sigma)^{n+s}\tau^{p-1}(1-\tau)^{q-1}}
  {((1-x)(1-y)+\sigma\tau x(1-y)+\sigma(1-\tau)y(1-x))^{s+1}} \,d\sigma\,d\tau\,dx\,dy .
\end{align*}
To~get hands on large $p,q$ asymptotics, we~make one more change of variable from $x,y,\sigma$
to $X,U,S$ via
\begin{align*}
& x=e^{-XU/q}, \quad y=e^{-(1-X)U/p}, \quad \sigma=e^{-S/(p+q)}, \\
&\hskip4em dx\,dy\,d\sigma = \frac{xy\sigma U}{pq(p+q)} \,dX\,dU\,dS,
\end{align*}
and also find it convenient to multiply both sides by $s+1$ and to introduce the function
$G(w):=\frac{1-e^{-w}}w$. This leads~to
\begin{align*}
& 2(s+1)c_{pq}(s) = \GG{p+q+n,s+2,p+n,q+n}{n,n,p,p,q,q,n+s+1} \int_0^\infty \int_0^\infty \int_0^1 \int_0^1 \\
&\;\times
 e^{-U-S} \Big[\frac{XU}q \frac{(1-X)U}p \frac S{p+q} G\Big(\frac{XU}q\Big) G\Big(\frac{(1-X)U}p\Big)
  G\Big(\frac S{p+q}\Big)\Big]^{n+s} \tau^{p-1}(1-\tau)^{q-1}  \\
&\;\times
 \Big[ \frac{XU}q \frac{(1-X)U}p G\Big(\frac{XU}q\Big) G\Big(\frac{(1-X)U}p\Big)
   +\tau e^{-\frac S{p+q}-\frac{XU}q} \frac{(1-X)U}p G\Big(\frac{(1-X)U}p\Big) \\
&\hskip4em
   +(1-\tau) e^{-\frac S{p+q}-\frac{(1-X)U}p} \frac{XU}q G\Big(\frac{XU}q\Big) \Big]^{-s-1}
   \, \frac{U\,d\tau\,dX\,dU\,dS}{pq(p+q)} .
\end{align*}

Now we specialize to the situation of the theorem, i.e.~take $p=P\lambda$, $q=(1-P)\lambda$,
with $0<P<1$ fixed and $\lambda\to+\infty$. This yields the huge formula
\[ \begin{aligned}
& 2(s+1)c_{pq}(s) = \GG{p+q+n,s+2,p+n,q+n}{n,n,p,p,q,q,n+s+1} \int_0^\infty \int_0^\infty \int_0^1 \int_0^1 \\
&\hskip2em
 e^{-U-S} \frac{X^{n+s}(1-X)^{n+s}S^{n+s}}{[P(1-P)\lambda^3]^{n+s+1}}
 \Big[G\Big(\frac{U_1}{\lambda}\Big) G\Big(\frac{U_2}{\lambda}\Big) G\Big(\frac S\lambda\Big)\Big]^{n+s} \\
&\hskip-2em
 \frac{(\tau^P(1-\tau)^{1-P})^\lambda / (\tau (1-\tau)) U^{2n+2s+1} \lambda^{s+1} \,d\tau\,dX\,dU\,dS }
  {\Big[ \frac{U_1U_2}{\lambda} G\Big(\frac{U_1}{\lambda}\Big) G\Big(\frac{U_2}{\lambda}\Big)
   +\tau e^{-\frac S\lambda-\frac{U_1}{\lambda}} U_2 G\Big(\frac{U_2}{\lambda}\Big)
   +(1-\tau) e^{-\frac S\lambda-\frac{U_2}{\lambda}} U_1 G\Big(\frac{U_1}{\lambda}\Big) \Big]^{s+1}} ,
 \end{aligned}
\label{tYD} \]
where for typographical reasons we have momentarily introduced the notations $U_1:=\frac{XU}{1-P}$, $U_2=\frac{(1-X)U}P$.
Let $g_k$ momentarily stand for the Taylor coefficients of the (entire) function $G(w)^{n+s}$:
$$ G(w)^{n+s} =: \sum_{k=0}^\infty g_k w^k, \qquad g_0=1.  $$
Then
\[ \Big[G\Big(\frac{U_1}{\lambda}\Big) G\Big(\frac{U_2}{\lambda}\Big) G\Big(\frac S\lambda\Big)\Big]^{n+s}
 = \sum_{j,k,l} g_j g_k g_l \frac{U_1^j U_2^k S^l} {\lambda^{j+k+l}} .  \label{tYE}  \]
Also,
\[ \begin{aligned}
& %\hskip2em 
\frac{U_1U_2}{\lambda} G\Big(\frac{U_1}{\lambda}\Big) G\Big(\frac{U_2}{\lambda}\Big) 
% \\ &\hskip2em 
  +\tau e^{-\frac S\lambda-\frac{U_1}{\lambda}} U_2 G\Big(\frac{U_2}{\lambda}\Big) 
% \\ &\hskip2em 
 +(1-\tau) e^{-\frac S\lambda-\frac{U_2}{\lambda}} U_1 G\Big(\frac{U_1}{\lambda}\Big) \\
& = \Big[\tau U_2+(1-\tau)U_1\Big]
 \Big[1+\sum_{m=1}^\infty \lambda^{-m} \frac{p_{m+1}(U_1,U_2,S,\tau)}
   {\tau U_2+(1-\tau)U_1} \Big] , \end{aligned} \label{tYF}  \]
where $p_{m+1}$ is a polynomial homogeneous of degree $m+1$ in $U_1,U_2,S$ and of degree 1 in~$\tau$.
Feeding \eqref{tYE} and~\eqref{tYF} into~\eqref{tYD}, we~arrive~at
\[ \begin{aligned}
& 2(s+1)c_{pq}(s) = \GG{p+q+n,s+2,p+n,q+n}{n,n,p,p,q,q,n+s+1} \int_0^\infty \int_0^\infty \int_0^1 \int_0^1 \\
&\hskip2em \frac{e^{-U-S} [X(1-X)S]^{n+s} U^{2n+2s+1} } {[P(1-P)]^{n+s+1}\lambda^{3n+2s+2}\tau(1-\tau) [\tau U_2+(1-\tau)U_1]^{s+1}} \\
&\hskip2em \sum_{j=0}^\infty \frac{a_{2j}(U_1,U_2,S,\tau)} {\lambda^j((1-\tau)U_1+\tau U_2)^j} (\tau^P(1-\tau)^{1-P})^\lambda
 \,d\tau\,dX\,dU\,dS, \end{aligned} \label{tYG}  \]
where $a_{2j}$ is a polynomial homogeneous of degree $2j$ in $U_1,U_2,S$ and of degree at most $j$ in~$\tau$, $a_0\equiv1$.
Carrying out the $X,S$ and $U$ integrations in \eqref{tYG} already produces the desired asymptotic expansion
in decreasing powers of~$\lambda$; it~only remains to treat the $\tau$ integral. To~do that, we~recall the
standard asymptotics of a Laplace integral, see e.g.~\cite[\S2.1, Theorem~1.3]{Fed}: if~$f,\cS$ are smooth
real-valued functions on a finite interval $[a,b]$, and $\cS$ attains its maximum at a unique point
$\tau_0\in(a,b)$, with $\cS''(\tau_0)\neq0$, then the Laplace integral
$$ \mathcal I(\lambda) = \int_a^b f(\tau) e^{\lambda\cS(\tau)} \,d\tau  $$
possesses the following asymptotics as $\lambda\to+\infty$:
$$ \mathcal I(\lambda) \approx e^{\lambda\cS(\tau_0)} \sum_{k=0}^\infty c_k \lambda^{-k-\frac12},  $$
with
$$ c_k = \frac{\Gamma(k+\frac12)}{(2k)!} \frac{d^{2k}}{d\tau^{2k}} [f(\tau)\cS(\tau,\tau_0)^{-k-\frac12}]_{\tau=\tau_0},  $$
where $\cS(\tau,\tau_0):=\frac{\cS(\tau_0)-\cS(\tau)}{(\tau-\tau_0)^2/2}$; in~particular,
$$ c_0 = f(\tau_0) \sqrt{\frac{2\pi}{-\cS''(\tau_0)}}.  $$
Applying this to~\eqref{tYG}, taking for $\cS$ the function $\cS(\tau)=P\log\tau+(1-P)\log(1-\tau)$,
with $\tau_0=P$ (so,~in~particular, $\cS(\tau,\tau_0)=\sum_{j=0}^\infty \frac{2(\tau-P)^j}{j+2}
((1-P)^{-j-1}-(-P)^{-j-1})$), and for $f$ all the remaining terms of the integrand, we~get
\begin{align*}
& 2(s+1)c_{pq}(s) = \GG{p+q+n,s+2,p+n,q+n}{n,n,p,p,q,q,n+s+1} (P^P(1-P)^{(1-P)})^\lambda  \\
&\hskip2em \int_0^\infty \int_0^\infty \int_0^1 \frac{e^{-U-S} [X(1-X)S]^{n+s} U^{2n+2s+1} }
  {[P(1-P)]^{n+s+1}\lambda^{3n+2s+2}}  \\
&\hskip2em \sum_{j,k=0}^\infty \GG{k+\frac12}{2k+1} \frac{d^{2k}}{d\tau^{2k}}
   \Big[\frac{a_{2j}(U_1,U_2,S,\tau)}
    {\tau(1-\tau) [(1-\tau)U_1+\tau U_2]^{j+s+1}} \cS(\tau,P)^{-k-\frac12} \Big]_{\tau=P}
   \lambda^{-k-\frac12-j} \,dX\,dU\,dS  \\
&= \GG{p+q+n,s+2,p+n,q+n}{n,n,p,p,q,q,n+s+1} \sqrt\pi (P^P(1-P)^{(1-P)})^\lambda \\
&\hskip2em \sum_{j,k=0}^\infty \lambda^{-3n-2s-\frac52-j-k}
 \int_0^\infty \int_0^\infty \int_0^1 \frac{e^{-U-S} [X(1-X)S]^{n+s} U^{2n+s-j}} {[P(1-P)]^{n+s+\frac32}}
  \Phi_{jk}(X,\tfrac1P,\tfrac1{1-P},U,S) \,dX\,dU\,dS
\end{align*}
with some polynomials $\Phi_{jk}$ in the indicated variables, $\Phi_{00}\equiv1$.
Carrying out the $S,U$ and $X$ integrations, we~obtain
\[ \begin{aligned}
2(s+1)c_{pq}(s) &= \GG{p+q+n,s+2,p+n,q+n}{n,n,p,p,q,q,n+s+1} \frac{\sqrt\pi(P^P(1-P)^{1-P})^\lambda}
  {[P(1-P)]^{n+s+\frac32}} \\
&\hskip4em  \lambda^{-3n-2s-\frac52} \sum_{m=0}^\infty b_m(\tfrac1P,\tfrac1{1-P}) \lambda^{-m},
\end{aligned} \label{tYH}  \]
with some polynomials $b_m$, $b_0\equiv\frac{\Gamma(n+s+1)^3\Gamma(2n+s+1)}{\Gamma(2n+2s+2)}$. 
(All~of $p_{m+1}$, $a_{2j}$, $\Phi_{jk}$ and $b_m$ depends also on~$s$, although this is not 
explicitly reflected by the notation.) Now~we employ the fact that
$2(s+1)c_{pq}(s)\to1$ as $s\searrow1$ (since the measures $(s+1)(1-t)^s\,dt$ converge weakly to
the Dirac mass at $t=1$); dividing both sides by~\eqref{tYH} by the same expressions with $s=-1$,
and restoring the variables $p=P\lambda$, $q=(1-P)\lambda$, we~finally arrive~at
$$ (s+1)c_{pq}(s) \approx \GG{2n+s+1,n+s+1,n+s+1,s+2}{n,n,2n+2s+2} (pq)^{-s-1} \Big[ 1+\sum_{j=1}^\infty A_j(p,q) \Big], $$
with some functions $A_j$ homogeneous of degree~$-j$. But~this is precisely~\eqref{tYC},
completing the proof.   %\qed
\end{proof}

Combining the last theorem with \eqref{tTR} from the Introduction, we~see that
$c_{pq}(s)\approx(p+1)^{-s-1}(q+1)^{-s-1}$ as $p+q\to+\infty$ (for fixed~$s$).
As~argued in the beginning of this section, the ``leading order'' term of $K_s(z,w)$
can thus be expected to be the same as for the function
\[ F_s(r\zeta,R\eta) := \sum_{p,q} S^{pq}(r) S^{pq}(R) H^{pq}(\cdd\zeta\eta)
 (p+\tfrac{n-1}2)^{s+1} (q+\tfrac{n-1}2)^{s+1}.  \label{tYI}  \]
We~conclude this section by discussing the boundary singularity of~$F_s$.

\begin{theorem} \label{PH}
For $n>1$ and $s=0,1,2,\dots$
\[ \begin{aligned}
&F_s(z,w) = \LL^{s+1} \Big[ \frac{\Gamma(n)}{2\pi^n} \frac{(1-y_2)^n}{(1-x_2)^n(1-y_1)^n}
 \sum_{i_1=0}^n \sum_{i_2,j_1=0}^{n-i_1}
 \frac{(-n)_{i_1+i_2} (-n)_{i_1+j_1} (n)_{i_2} (n)_{j_1}} {i_1!i_2!j_1!(n)_{i_1+i_2+j_1}} \\
&\hskip0em x_1^{i_1} \Big(\frac{x_1-y_1}{1-y_1}\Big)^{i_2} \Big(\frac{x_1-x_2}{1-x_2}\Big)^{j_1}
 \FF21{i_2+n,j_1+n}{i_1+i_2+j_1+n}{1-\frac{(1-x_1)(1-y_2)}{(1-x_2)(1-y_1)}} \Big]
 \Big|_{\substack{x_1=|z|^2, x_2=\spr{z,w},\\ y_1=\spr{w,z}, y_2=|w|^2}}, \end{aligned}  \label{tYJ}  \]
where $\LL$ is the linear differential operator
\[ \LL:=(x_2y_1-x_1y_2)\frac{\partial^2}{\partial x_2\partial y_1}
 +\frac{n-1}2 \Big(x_2\frac\partial{\partial x_2}+y_1\frac\partial{\partial y_1}\Big)
 +\frac{(n-1)^2}4 I.  \label{tYK}  \]
\end{theorem}

\begin{proof} Consider the differential operator
$$ D := -\frac{\dsph}4-\frac{\cR^2}4 + \frac{(n-1)^2}4 I.  $$
By~\eqref{tYA},
$$ D|\bhpq = (p+\tfrac{n-1}2) (q+\tfrac{n-1}2) I | \bhpq .  $$
Consequently, using~\eqref{tVJ},
\[ F_s(z,w) = D^{s+1} K\Sz(z,w),  \label{tYL}   \]
with the understanding that, to~fix ideas, $D$~is always applied to the $z$ variable.
Note that since $\dsph$ and $\cR$ are ``tangential'' operators --- that~is, they act only
on the $\zeta$ variable in the polar coordinates $z=r\zeta$ --- we~have $D(fg)=fDg$ for
any function $f$ depending on~$|z|$ only. Substituting \eqref{tTS} for the $K\Sz$ in~\eqref{tYL},
we~thus obtain
\begin{align*}
F_s(z,w) &= \frac{\Gamma(n)}{2\pi^n} (1-|z|^2)^n (1-|w|^2)^n \sum_{p,q,j,m=0}^\infty
 \frac{(n)_{p+j}(n)_{q+j}(n)_{p+m}(n)_{q+m}}{(n)_{p+q+j+m}}  \\
&\hskip2em\times D^{s+1} \frac{\spr{z,w}^p\spr{w,z}^q|z|^{2j}|w|^{2m}}{p!q!j!m!} .
\end{align*}
Now~by direct computation
\begin{align*}
& D[\spr{z,w}^p\spr{w,z}^q|z|^{2j}|w|^{2m}] \\
&\hskip4em = \Big[ \Big(1-\frac{|z|^2|w|^2}{|\spr{z,w}|^2}\Big) pq + \frac{n-1}2(p+q) + \Big(\frac{n-1}2\Big)^2 \Big]
 \spr{z,w}^p\spr{w,z}^q|z|^{2j}|w|^{2m} \\
&\hskip4em = \LL[x_1^j x_2^p y_1^q y_2^m]_{x_1=|z|^2,x_2=\spr{z,w},y_1=\spr{w,z},y_2=|w|^2} .
\end{align*}
Hence
$$ F_s(z,w) = \LL^{s+1} \Big[\frac{\Gamma(n)}{2\pi^n}(1-x_1)^n(1-y_2)^n \FD{n,n,n,n}n{x_1,x_2,y_1,y_2}\Big]
 _{\substack{x_1=|z|^2,x_2=\spr{z,w},\\y_1=\spr{w,z},y_2=|w|^2}} .  $$
Replacing the expression in the square brackets by the one from Corollary~\ref{PC} yields the desired claim. %\qed
\end{proof}

\begin{corollary} \label{PI}
For $n>1$ and $s=0,1,2,\dots$, $F_s\in C^{n-1}(\bbd)$.
\end{corollary}

\begin{proof} We~have seen in course of the proof of Proposition~\ref{PE}, cf.~\eqref{tWT}, that
\[ K\Sz(z,w) = \sum_{j,k=0}^\infty a_{jk}(Q,\overline Q)(1-x_1)^j(1-y_2)^k
 + \sum_{j,k=0}^\infty b_{jk}(Q,\overline Q)(1-x_1)^j(1-y_2)^k \log[(1-x_1)(1-y_2)],  \label{tYM} \]
where we have again set $x_1=|z|^2$, $x_2=\spr{z,w}$, $y_1=\spr{w,z}$, $y_2=|w|^2$ and
$Q:=1/(1-x_2)$, $\overline Q=1/(1-y_1)$, and $a_{jk},b_{jk}$ are holomorphic functions of
$Q,\overline Q$ in the right half-plane $\Re Q>0$. Now~since $\LL$ involves only differentiations
with respect to the $x_2$ and $y_1$ variables, the application of $\LL^{s+1}$ to \eqref{tYM}
preserves this form of the right-hand side (only the functions $a_{jk},b_{jk}$ will get changed).
Since, again as in the proof of Proposition~\ref{PE}, the first summand is $C^\infty$ on
$|1-x_1|<1$, $|1-y_2|<1$ and $\Re Q>0$, while the second summand is $C^{n-1}$ there,
the~claim follows.  %\qed
\end{proof}

For $z,w\in\Bn$, the three quantities
$$ x_1=|z|^2, \quad x_2=\overline y_1=\spr{z,w}, \quad y_2=|w|^2,  $$
are preserved upon replacing $z,w$ by $Uz,Uw$ with any $U\in\Un$, and satisfy
\[ 0\le x_1,y_2<1, \qquad x_2\in\CC, \; |x_2|^2\le x_1 y_2 .  \label{tYN}  \]
Conversely, all triples $x_1,x_2,y_2$ satisfying \eqref{tYN} arise in this way,
and if $z',w'\in\Bn$ give rise to the same triple $x_1,x_2,y_2$ as $z,w$,
then there is some $U\in\Un$ for which $z'=Uz$ and $w'=Uw$. In~other words,
the~map $(z,w)\mapsto(|z|^2,\spr{z,w},|w|^2)$ is a bijection of the equivalence
classes of $\Bn\times\Bn$ modulo the diagonal action of $\Un$ onto the set of
all triples $x_1,x_2,y_2$ satisfying~\eqref{tYN}.

Now instead of $z,w$, we~can apply the observation in the preceding paragraph to
the pair $z,\phi_zw$. Since
$$ \phi_{Uz}(Uw) = U(\phi_z w), \qquad\forall z,w\in\Bn, \; U\in\Un,  $$
it~again transpires that the map
\[ (z,w)\mapsto(U,V,Z), \qquad U:=|z|^2, \; V:=|\phi_zw|^2, \; Z:=\spr{z,\phi_zw} , \label{tYO}  \]
is~a~bijection of the equivalence classes of $\Bn\times\Bn$ modulo the diagonal
action of $\Un$ onto the set
\[ \Omega := \{(U,V,Z): 0\le U,V<1, \; Z\in\CC,\; |Z|^2\le UV \} .  \label{tYP}  \]
We~conclude by expressing the singularity of $F_s(z,w)$ at the boundary diagonal
in terms of the ``coordinates'' $(U,V,Z)$.

\begin{corollary} \label{PJ}
For $n>1$ and $s=0,1,2,\dots$,
$$ F_s(z,w) = \frac{(1-V)^n}{(1-U)^{n+s+1}}
\sum_{i_1=0}^n \sum_{i_2,j_1=0}^{n-i_1} \sum_{k=0}^{s+1}
 P_{i_1 i_2 j_1 k}(U,Z,\oZ,V) \FF21{i_2+m+k,j_1+n+k}{i_1+i_2+j_1+n+k}V, $$
where $P_{i_1 i_2 j_1 k}(U,Z,\oZ,V)$ is a polynomial of degree at most
$i_1+s+1$, $j_1+s+1$, $i_2+s+1$ and $k+s+1$, respectively, in~the indicated variables.
\end{corollary}

\begin{proof} We~still keep the previous notation $x_1=|z|^2$, $x_2=\spr{z,w}$, $y_1=\spr{w,z}$, $y_2=|w|^2$.
Taking $0$ and $w$ for $w_1$ and $w_2$, respectively, in~\eqref{tTT}, we~get
$$ Z=\spr{\phi_z0,\phi_zw} = 1-\frac{1-|z|^2}{1-\spr{z,w}} = \frac{x_1-x_2}{1-x_2},  $$
and similarly
$$ \oZ = \frac{x_1-y_1}{1-y_1}, \qquad V=1-\frac{(1-x_1)(1-y_2)}{(1-x_2)(1-y_1)}  $$
(which establishes~\eqref{tTQ}). In~terms of $U,V,Z$, \eqref{tWN}~therefore becomes simply
$$ K\Sz(z,w) = \frac{(1-V)^n}{(1-U)^n} \sum_{\substack{i_1+i_2\le n,\\i_1+j_1\le n}}
 a_{i_1 i_2 j_1} U^{i_1} \oZ^{i_2} Z^{j_1} \FF21{i_2+n,j_1+n}{i_1+i_2+j_1+n}V ,  $$
with $a_{i_1i_2j_1}:=\frac{\Gamma(n)}{2\pi^n}\frac{(-n)_{i_1+i_2}(-n)_{i_1+j_1}(n)_{i_2}(n)_{j_1}}
{i_1!i_2!j_1!(n)_{i_1+i_2+j_1}}$. On~the other hand, by~a~tedious but routine computation,
the operator $\LL$ expressed in the coordinates $U,V,Z$ takes the form
\begin{align*}
\LL &= \frac{|Z|^2-UV}{1-U} [|1-Z|^2\partial_Z\partial_\oZ
 + (1-\oZ)(1-V) \partial_\oZ \partial_V + (1-Z)(1-V) \partial_Z \partial_V
 + (1-V)^2\partial_V^2 - (1-V)\partial_V ]  \\
& - \frac{n-1}{2(1-U)} [(1-Z)(U-Z)\partial_Z+(1-\oZ)(U-\oZ)\partial_\oZ +(1-V)(2U-Z-\oZ)\partial_V]
 + \Big(\frac{n-1}2\Big)^2 I.
\end{align*}
Combining these two facts with the formula
\[ \partial_V \FF21{a,b}cV = \frac{ab}c \FF21{a+1,b+1}{c+1}V  \label{tYQ}  \]
for the derivative of a hypergeometric function, the~claim follows by a simple induction argument. %\qed
\end{proof}

The~advantage of the coordinates $U,V,Z$ is that when $(z,w)$ approaches the boundary diagonal ---
that~is, essentially, when both $z$ and $w$ approach the same point $\zeta\in\pBn$ --- then of course
all of $x_1=|z|^2$, $x_2=\spr{z,w}$, $y_1=\spr{w,z}$ and $y_2=|w|^2$ tend to~$1$, but $|\phi_zw|^2=V$ can
behave in many ways: or instance, for $z=w$ one has $V=0$, while e.g.~for $z=(1-h)e_1$, $w=(1-h^2)e_1$ one
has $V\to1$ as $h\searrow0$. Thus the coordinates $U,V,Z$ capture how $z$ approaches $\zeta$
``relative'' to~$w$. Of~course, the downside of the formula in the last corollary is that it
completely hides the symmetry $z\leftrightarrow w$. At~the moment, we~do not know how to express
the boundary singularity of $F_s$ in a manner that would make this symmetry evident.

\section{Concluding remarks} \label{sec6}
\subsection{Formulas for $c_{pq}$} The~coefficients $c_{pq}(s)$ can be expressed in terms of
various multivariable hypergeometric functions. One~such expression comes from using the
formula for the Taylor coefficients of the square of a ${}_2\!F_1$ function \cite[\S4.3(14)]{BE}
$$ \FF21{a,b}cz^2 = \sum_{m=0}^\infty \FF43{a,b,1-c-m,-m}{c,1-a-m,1-b-m}1
 \frac{(a)_m(b)_m}{(c)_mm!}z^m ,  $$
yielding
\[ \begin{aligned}
 2c_{pq}(s) &= \GG{n+p,n+q}{n,n+p+q}^2 \sum_{m=0}^\infty \FF43{p,q,1-n-p-q-m,-m}{n+p+q,1-p-m,1-q-m}1 \\
&\hskip4em \times \frac{(p)_m(q)_m}{(n+p+q)_m m!} \GG{n+p+q+m,s+1}{n+p+q+m+s+1}, \end{aligned}
 \label{tZG}  \]
which however is not very useful. A~somewhat nicer expression is obtained upon expanding both
${}_2\!F_1$ factors in \eqref{tXH} into Taylor series and integrating term by~term; this gives
\[ \begin{aligned}
2c_{pq}(s) &= \GG{n+p,n+q}{n,n+p+q}^2 \sum_{j,k=0}^\infty \frac{(p)_j(q)_j}{(n+p+q)_j j!}
 \frac{(p)_k(q)_k}{(n+p+q)_k k!} \GG{n+p+q+j+k,s+1}{n+p+q+j+k+s+1}   \\
&= \GG{n+p,n+q}{n,n+p+q}^2 \GG{n+p+q,s+1}{n+p+q+s+1} \\
&\hskip4em\times \sum_{j,k=0}^\infty
 \frac{(p)_j(q)_j}{(n+p+q)_j j!} \frac{(p)_k(q)_k}{(n+p+q)_k k!}
 \frac{(n+p+q)_{j+k}}{(n+p+q+s+1)_{j+k}} . \end{aligned}
 \label{tZH}  \]
In~terms of the higher order hypergeometric function of two variables of Appell and
Kamp\'e de F\'eriet \cite[p.~150]{AKF}
$$ F\vast( \begin{matrix} \mu \\ \nu \\ \rho \\ \sigma \end{matrix} \vast|
 \begin{matrix} a_1,\dots,a_\mu \\ b_1,b'_1,\dots,b_\nu,b'_\nu \\ c_1,\dots,c_\rho \\ d_1,d'_1,\dots,d_\sigma,d'_\sigma \end{matrix} \vast|
 x,y \vast) := \sum_{j,k=0}^\infty \frac{\prod_{i=1}^\mu (a_i)_{j+k}}{\prod_{i=1}^\rho(c_i)_{j+k}}
 \frac{\prod_{i=1}^\nu(b_i)_j(b'_i)_k}{\prod_{i=1}^\sigma(d_i)_j(d'_i)_k} \frac{x^j y^k}{j!k!}  $$
this becomes
\[ 2c_{pq}(s) = \GG{n+p,n+p,n+q,n+q,s+1}{n,n,n+p+q,n+p+q+s+1}
 F\vast( \begin{matrix} 1 \\ 2 \\ 1 \\ 1 \end{matrix} \vast|
 \begin{matrix} n+p+q \\ p,p,q,q \\ n+p+q+s+1 \\ n+p+q,n+p+q \end{matrix} \vast| 1,1 \vast) .  \label{tZI}  \]
In~the notation of Karlsson and Srivastava \cite[p.~27]{KS}, the last function is denoted
$F^{1:2,2}_{1:1,1}\Big(\begin{matrix}p+q+n & : & p,q;p,q;\\p+q+n+s+1&:&p+q+n;p+q+n;\end{matrix}1,1\Big)$.
However, an~explicit expression for the value of these seems again not to be available.
In~fact, our computations at the end of Section~\ref{sec4} amount to an evaluation of~\eqref{tZI}
for the special case of $n=2$, $s=0$ and $p=q=1$.

\subsection{Uniformity of asymptotic expansions} In~Theorem~\ref{PG}, we~have found the asymptotics
of $c_{\lambda P,\lambda(1-P)}(s)$ as $\lambda\to+\infty$, for fixed $s$ and fixed $0<P<1$;
also, the case of $P\in\{0,1\}$ has been handled separately. We~expect that the asymptotic expansion
obtained is actually uniform in $P\in[0,1]$, and hence in fact yields the unrestricted asymptotics
of $c_{pq}(s)$ as $p+q\to+\infty$ (with $s$ fixed); however, at~the moment we can offer no proof that
this is indeed the case. This kind of difficulty arises also in other situations where two-parameter
asymptotics are involved, cf.~e.g.~\S II.7 in~\cite{Fed}.

\subsection{The Wallach set} In~the holomorphic case, the weighted Bergman kernels $K\hol_s(z,w)$,
$s>-1$, on~$\Bn$ actually extend by analytic continuation to a meromorphic function of $s\in\CC$,
and continue to be positive defnite functions on $\Bn\times\Bn$ for all $s\in(-n-1,+\infty)$;
the~last interval is called the \emph{Wallach set} of~$\Bn$, and there is a similar story for $\Bn$
replaced by any irreducible bounded symmetric domain in $\CC^n$ \cite{RV}. More precisely, if~one first
normalizes the measures $(1-|z|^2)^s\,dz$ on $\Bn$ to be of total mass~one --- that~is, in~other words,
multiplies them by~$1/c_{00}(s)$ --- then the normalized reproducing kernels
\[ c_{00}(s) K\hol_s(z,w) = (1-\spr{z,w})^{-n-1-s}  \label{tZJ}  \]
extend to a holomorphic function of $z,\overline w\in\Bn$ and $s\in\CC$, and this analytic continuation
is still a positive definite kernel in $(z,w)$ on $\Bn\times\Bn$ as long as $s>-n-1$
(and only for these~$s$).

Similar phenomenon prevails for the harmonic case \cite{EJMPA}. We~show that the \Mh case~is,
likewise, no~exception.

\begin{proposition} \label{PK}
The normalized \Mh weighted Bergman kernels
$$ c_{00}(s) K_s(z,w)  $$
extend by analytic continuation in $s$ to a meromorphic function on $\Bn\times\Bn\times\CC$,
and remain positive definite in $(z,w)$ as long as $s>-n-1$ (and only for such~$s$).

Thus, the ``\Mh Wallach set'' for $\Bn$ is the interval $(-n-1,+\infty)$.
\end{proposition}

\begin{proof} By~\eqref{tXD}, we~have
\[ c_{00}(s) K_s(r\zeta,R\eta) = \sum_{p,q} \frac{c_{00}(s)}{c_{pq}(s)} S^{pq}(r)S^{pq}(R)H^{pq}(\cdd\zeta\eta). \label{tZA}  \]
It~is therefore enough to exhibit an analytic continuation of the functions $f_{pq}(s):=\frac{c_{00}(s)}{c_{pq}(s)}$ in~$s$
for all~$p,q$, and show that the ``Wallach set''
$$ \mathcal W:=\{s\in\CC: 0<f_{pq}(s)<+\infty \; \forall p,q\ge0 \}  $$
coincides with the interval $(-n-1,+\infty)$. A~routine convergence check then yields the desired
analytic continuation of~\eqref{tZA}, and we are done.

First of all, from either \eqref{tZG} or \eqref{tZH} it is immediate that $c_{pq}(s)$ and, hence,
$f_{pq}(s)$ extend meromorphically to all $s\in\CC$ (because the Gamma function does).

Next, we~have already seen in \eqref{tXK} that
$$ 2c_{00}(s) = \GG{n,s+1}{n+s+1} = \frac{\Gamma(n)}{(s+1)_n} .  $$
Similarly
$$ 2c_{10}(s) = \GG{n+1,s+1}{n+s+2}.  $$
Hence
$$ f_{10}(s) = \frac{c_{00}(s)}{c_{10}(s)} = \GG{n,s+1,n+s+2}{n+s+1,n+1,s+1} = \frac{n+s+1}n.  $$
This is positive only for $n+s+1>0$, so $\mathcal W\subset(-n-1,+\infty)$.

Let us momentarily denote
\[ G_{pq}(t) := \GG{n+p,n+q}{n,n+p+q}^2 t^{p+q+n-1} \FF21{p,q}{n+p+q}t ^2 ,  \label{tZB}  \]
so~that
\[ 2c_{pq}(s) = \int_0^1 G_{pq}(t) (1-t)^s \,dt, \qquad \Re s>-1.  \label{tZF}  \]
Integrating by parts $n$ times, we~obtain
\[ 2c_{pq}(s) = \sum_{j=0}^{n-1} \frac{[-G_{pq}^{(j)}(t)(1-t)^{s+j+1}]_{t=0}^{t=1}}{(s+1)_{j+1}}
 + \frac1{(s+1)_n} \int_0^1 G_{pq}^{(n)} (t) (1-t)^{s+n} \, dt ,  \label{tZC}  \]
provided all the terms are finite. Now~by~\eqref{tZB}, the derivative $G_{pq}^{(k)}(t)$
is continuous up to $t=1$ for $k<n$, while $G_{pq}^{(n)}(t)\approx\log\frac1{1-t}$.
Also, $G_{pq}^{(j)}(0)=0$ whenever $j<p+q+n-1$. Thus all the terms in the sum in~\eqref{tZC}
actually vanish for $s>-n-1$, and we thus get for any $(p,q)\neq(0,0)$ and $s>-n-1$
the representation
$$ 2c_{pq}(s) = \frac1{(s+1)_n} \int_0^1 G_{pq}^{(n)}(t) (1-t)^{s+n} \,dt  $$
and
\[ \frac{c_{pq}(s)}{c_{00}(s)} = \frac1{\Gamma(n)} \int_0^1 G_{pq}^{(n)}(t) (1-t)^{s+n} \,dt. \label{tZD}  \]
Furthermore, in~view of~\eqref{tYQ}, all derivatives of $G_{pq}$ are positive on~$(0,1)$,
hence the integrand in \eqref{tZD} is positive, and $f_{pq}(s)$ is positive and finite,
for~all $s>-n-1$. Consequently, $(-n-1,+\infty)\subset\mathcal W$, completing the proof.  %\qed
\end{proof}

\subsection{Semiclassical asymptotics}
In~the holomorphic case, the behavior of $K_s(z,w)$ as $s\to+\infty$ is of importance in certain
quantization procedures~\cite{E24}; the~weight parameter $s$ plays the role of the reciprocal of
the Planck constant, and one therefore speaks of ``semiclassical'' limits. The~following result
is again the same as for the holomorphic, as~well as for some harmonic~\cite{EIceland}, situations.

\begin{proposition} \label{PL}
For all $z\in\Bn$,
$$ \lim_{s\to+\infty} K_s(z,z)^{1/s} = (1-|z|^2)^{-1}.  $$
\end{proposition}

\begin{proof} From the definition of the reproducing kernel $K_\HH(z,z)$ as the norm square of
the evaluation functional,
\[ K_\HH(z,z) = \sup\{|f(z)|^2: f\in\HH,\;\|f\|_\HH\le1 \},  \label{tZK}  \]
and the fact that holomorphic functions are \Mh it follows that
$$ K_s(z,z) \ge K\hol_s(z,z) .  $$
Hence
\[ \liminf_{s\to+\infty} K_s(z,z)^{1/s} \ge \liminf_{s\to+\infty} K\hol_s(z,z)^{1/s} = (1-|z|^2)^{-1} , \label{tZL} \]
by~the known result for the holomorphic case (cf.~\eqref{tZJ}).

On~the other hand, by the invariant mean value property of \Mh functions,
we~have for any \Mh $f$ and $0<r<1$
\begin{align*}
f(z) = (f\circ\phi_z)(0) &= \frac{n!}{\pi^n r^{2n}} \int_{|x|<r} (f\circ\phi_z)(x) \,dx  \\
&= \frac{n!}{\pi^n r^{2n}} \int_{|\phi_z y|<r} f(y) \, J_z(y) \,dy,
\end{align*}
where $J_z$ is the Jacobian of~$\phi_z$. By~Cauchy-Schwarz,
\begin{align*}
|f(z)| &\le \frac{n!}{\pi^n r^{2n}} \Big(\int_{|\phi_z y|<r} |f(y)|^2 (1-|y|^2)^s \,dy\Big)^{1/2}
  \Big(\int_{|\phi_z y|<r} \frac{J_z(y)^2}{(1-|y|^2)^s} \,dy\Big)^{1/2}  \\
&\le \frac{n!}{\pi^n r^{2n}} \|f\|_s \Big(\int_{|\phi_z y|<r} \frac{J_z(y)^2}{(1-|y|^2)^s} \,dy\Big)^{1/2} .
\end{align*}
Taking supremum over all $f$ with $\|f\|_s\le1$, we~get by~\eqref{tZK}
$$ K_s(z,z)^{1/2} \le \frac{n!}{\pi^n r^{2n}} \Big(\int_{|\phi_z y|<r} \frac{J_z(y)^2}{(1-|y|^2)^s} \,dy\Big)^{1/2}, $$
whence
$$ K_s(z,z)^{1/s} \le \Big(\frac{n!}{\pi^n r^{2n}}\Big)^{2/s} \Big(\int_{|\phi_z y|<r} \frac{J_z(y)^2}{(1-|y|^2)^s} \,dy\Big)^{1/s} . $$
Letting $s\to+\infty$ and using the fact that
$$ \|F\|_{L^s(d\mu} \to \|F\|_{L^\infty(d\mu)} \qquad\text{as }s\to+\infty  $$
for any finite measure $\mu$ and bounded function~$F$, we~obtain
$$ \limsup_{s\to+\infty} K_s(z,z)^{1/s} \le \sup_{|\phi_z(y)|<r} (1-|y|^2)^{-1}.  $$
Finally, since $r\in(0,1)$ was arbitrary, letting $r\searrow0$ yields
$$ \limsup_{s\to+\infty} K_s(z,z)^{1/s} \le (1-|z|^2)^{-1}.  $$
Combining this with \eqref{tZL} completes the proof.  %\qed
\end{proof}

It~would be interesting to know what is the limit of $|K_s(z,w)|^{1/s}$ as $s\to+\infty$
for $z\neq w$, or~whether there is any asymptotic expansion of $K_s(z,z)(1-|z|^2)^s$ in
decreasing powers of~$s$, as~is the case for some harmonic Bergman kernels~\cite{EIceland}.

\subsection{Particular cases} For $z\bot w$, the formula \eqref{tTH} for the Szeg\"o kernel
$K\Sz(z,w)$ can also be evaluated in a different~way: namely, using $\Un$-invariance,
we~can assume without loss of generality that $z=|z|e_1$ and $w=|w|e_2$. The~integral in
\eqref{tTH} then takes the form
$$ (1-|z|^2)^n (1-|w|^2)^n \int_\pBn |1-|z|\zeta_1|^{-2n} |1-|w|\zeta_2|^{-2n} \,d\sigma(\zeta).  $$
Using the binomial expansion for $(1-t)^{-n}$ and integrating term by term shows that
the last integral equals
\[ \frac{2\pi^n}{\Gamma(n)} \sum_{k,l=0}^\infty \frac{(n)_k^2(n)_l^2}{(n)_{k+l}} \frac{|z|^{2k}|w|^{2l}}{k!l!}
= \frac{2\pi^n}{\Gamma(n)} \FE3{n,n,n,n}n{|z|^2,|w|^2} , \label{tZM}  \]
where $F_3$ is the third Appell hypergeometric function
$$ \FE3{a,a',b,b'}c{x,y} = \sum_{j,k} \frac{(a)_j(a')_k (b)_j (b')_k}{(c)_{j+k}j!k!} x^j y^k .  $$
Naturally, \eqref{tZM} agrees with the formula \eqref{tTS} for $\spr{z,w}=\spr{w,z}=0$, as~it~should.

Another case when $K\Sz(z,w)$ can be evaluated independently is when $z=w$. The~integral in \eqref{tTH}
then becomes
$$ (1-|z|^2)^{2n} \int_\pBn |1-\spr{z,\zeta}|^{-4n} \,d\sigma(\zeta). $$
Upon substituting again the binomial expansion and integrating term by term, this produces
\begin{align}
& \frac{2\pi^n}{\Gamma(n)} (1-|z|^2)^{2n} \sum_k \frac{(2n)_k^2}{(n)_k} \frac{|z|^{2k}}{k!}
 = \frac{2\pi^n}{\Gamma(n)} (1-|z|^2)^{2n} \FF21{2n,2n}n{|z|^2} \nonumber  \\
&= \frac{2\pi^n}{\Gamma(n)} (1-|z|^2)^{-n} \FF21{-n,-n}n{|z|^2} \label{tZN}
\end{align}
(note that the last ${}_2\!F_1$ is actually a polynomial; we~have used the Euler transformation
formula~\eqref{tWV}).

It~is entertaining to compare \eqref{tZN} with direct application of the formula~\eqref{tXD},
where in the latter the Szeg\"o case corresponds, as~we have already noted repeatedly,
to $c_{pq}\equiv1$ $\forall p,q$. Namely, \eqref{tXD}~gives
$$ K\Sz(z,z) = \sum_{p,q} S^{pq}(|z|)^2 H^{pq}(1).  $$
Now~by~\eqref{tVC}
$$ H^{pq}(1) = \frac{(n+p+q-1)(n+p-2)!(n+q-2)!}{p!q!(n-1)!(n-2)!}. $$
Substituting \eqref{tVF} for~$S^{pq}$, we~thus get from~\eqref{tZN}
$$ \sum_{p,q} t^{p+q} \GG{n+p,n+q}{n,n+p+q}^2  \FF21{p,q}{n+p+q}t ^2 H^{pq}(1) = (1-t)^{-n} \FF21{-n,-n}nt, $$
or, using again~\eqref{tWV},
$$ \sum_{p,q} t^{p+q} \GG{n+p,n+q}{n,n+p+q}^2  \FF21{n+p,n+q}{n+p+q}t ^2 H^{pq}(1) = \FF21{2n,2n}nt. $$
Looking at the coefficients at like powers of~$t$, this is equivalent~to
$$ \sum_{\substack{p,q,j,k\ge0\\p+q+j+k=m}} \GG{p+n+j,q+n+j}{n,p+q+n+j} \GG{p+n+k,q+n+k}{n,p+q+n+k} H^{pq}(1) = \frac{(2n)_m^2}{m!(n)_m}, $$
for all $m=0,1,2,\dots$.
We~do not know a direct proof of this (valid) formula.

\subsection{Weighted \Mh Green functions} It~has been known since the monograph of Bergman and Schiffer
\cite{BS} that the harmonic Bergman kernel $H(z,w)$ on a domain is intimately connected with the
Green function $G(z,w)$ for the biharmonic operator $\Delta^2$ with Dirichlet boundary conditions:
namely,
$$ H(z,w) = - \Delta_z \Delta_w G(z,w) ,  $$
where the subscript at $\Delta$ indicates the variable that the operator is applied~to.
An~analogous formula holds also for the weighted case. We~conclude this paper by pointing
out a similar connection in the \Mh case.

Let
$$ d\tau(z) := (1-|z|^2)^{-n-1}\,dz  $$
be the invariant measure on~$\Bn$. Given a positive $C^\infty$ weight functions $\rho$ on~$\Bn$,
consider the ``invariant weighted biharmonic'' operator
$$ \td \rho \td $$
with our $\td$ from~\eqref{tTA}. By~its Green function at a point $w\in\Bn$ we mean,
by~definition, a~function $G(z,w)\equiv G_w(z)$ such that, firstly,
$$ u(w) = \int_\Bn G_w \td\rho\td u \, d\tau $$
for all smooth functions $u$ whose support is a compact subset of~$\Bn$
(that~is, $\td\rho\td G_w=\delta_w$, the Dirac point mass at~$w$ with respect to~$d\tau$,
in~the sense of distributions); and secondly, both $G_w$ and its normal derivative $\partial G_w/
\partial n$ vanish at~$\pBn$. (More precisely --- they grow sufficiently slowly as~$|z|\nearrow1$;
we~are skipping some technical details here.)
It~is then a fact that such $G_w$ exists, is~uniquely determined, $G(z,w)=G(w,z)$,
and the following proposition holds.

\begin{proposition}  \label{PM}
The weighted \Mh Bergman kernel on $\Bn$ with respect to the measure $\rho(z)^{-1}\,d\tau(z)$
satisfies
\[ K_{d\tau/\rho}(z,w) = - \rho(z) \rho(w) \td_z \td_w G(z,w).  \label{tZO}  \]
\end{proposition}

We~omit the proof, which is the same as for the original harmonic case in~\cite{BS},
only using the Green formula for $d\tau$ from \cite[Section~1.6]{ZhuBALL} in the place
of the ordinary Green formula.

Choosing, in~particular, $\rho(z)=(1-|z|^2)^{-n-s-1}$, the \Mh kernels $K_{d\tau/\rho}$ in
\eqref{tZO} will be precisely our~$K_s$. One~of our earlier ideas how to compute $K_s(z,w)$
was to find $G(z,w)$ first and then apply~\eqref{tZO}; a~possible approach to finding $G(z,w)$
being along the lines of the one used in \cite{EP1} for the ordinary biharmonic Green function
on the annulus, and in \cite{EP2} for the unweighted invariant biharmonic Green function on the~ball.
In~both cases, the~ordinary decomposition of a function into its Fourier components would need
to be replaced by the decomposition \eqref{tVB} into the $(p,q)$ components with respect to
the action of~$\Un$. Unfortunately, so~far we have not been able to carry out this program.

\end{document}